\documentclass[12pt,a4paper,reqno]{amsart}
\usepackage{a4wide,amsfonts,amssymb,mathptmx}
\usepackage{hyperref}

\newtheorem{fact}{{\bf Fact}}

\providecommand{\BBb}[1]{{\mathbb{#1}}}

\providecommand{\cal}[1]{{\mathcal{#1}}}   

\newcommand{\B}{{\BBb B}}
\newcommand{\C}{{\BBb C}}
\newcommand{\dual}[2]{\langle\,#1,\,#2\,\rangle}
\newcommand{\im}{\operatorname{i}}
\renewcommand{\Im}{\operatorname{Im}}
\newcommand{\lap}{\operatorname{\Delta}}
\newcommand{\mlap}{-\!\operatorname{\Delta}}
\newcommand{\op}[1]{\operatorname{#1}}
\newcommand{\N}{\BBb N}
\newcommand{\point}{\operatorname{point}}     
\renewcommand{\Re}{\operatorname{Re}}

\newcommand{\R}{{\BBb R}}
\newcommand{\Rn}{{\BBb R}^{n}}
\newcommand{\Span}{\operatorname{span}}
\newcommand{\supp}{\operatorname{supp}}

\newtheorem{Theorem}{Theorem}
\newtheorem{Proposition}{Proposition}
\newtheorem{Lemma}{Lemma}
\newtheorem{Corollary}{Corollary}
\theoremstyle{definition}
\newtheorem{Example}{Example}
\theoremstyle{remark}
\newtheorem{Remark}{Remark}

\newcommand{\1}{{\ensuremath{\mathbf{1}}}}

\newcommand{\D}{\mathcal{D}}

\newcommand{\A}{\mathcal{A}}
\newcommand{\Gen}{\mathbf{A}}
\newcommand{\vvvert}{{|\hspace{-1.6pt}|\hspace{-1.6pt}|}}

\newcommand{\ip}[2]{\ensuremath{({#1}\,|\,{#2})}} 
\newcommand{\dualp}[2]{\ensuremath{\langle{#1},{#2}\rangle}} 

\def\Xint#1{\mathchoice
{\XXint\displaystyle\textstyle{#1}}%
{\XXint\textstyle\scriptstyle{#1}}%
{\XXint\scriptstyle\scriptscriptstyle{#1}}%
{\XXint\scriptscriptstyle\scriptscriptstyle{#1}}%
\!\int}
\def\XXint#1#2#3{{\setbox0=\hbox{$#1{#2#3}{\int}$}
\vcenter{\hbox{$#2#3$}}\kern-.5\wd0}}

\def\dashint{\Xint-}

\begin{document}
\title[Final value problems and well-posedness]{%
Final value problems for parabolic differential equations\\ and their well-posedness}
\author[Christensen and Johnsen]{A.-E.~Christensen and J.~Johnsen%
}
\address{%
Department of Mathematics, Aalborg University,
Skjernvej 4A, DK-9220 Aalborg {\O}st, Denmark%
}
\email{jjohnsen@math.aau.dk}
\begin{abstract} 
  This article concerns the basic understanding of parabolic final
  value problems, and a large class of such problems is proved to be well posed.
  The clarification is obtained via explicit Hilbert spaces that characterise the possible data,
  giving existence, uniqueness and stability of the corresponding solutions.  
  The data space is given as the graph normed domain of an unbounded operator occurring naturally
  in the theory. It induces a new compatibility
  condition, which relies on the fact, shown here, that analytic semigroups always are invertible
  in the class of closed operators. The general set-up is evolution equations for
  Lax--Milgram operators in spaces of vector distributions.
  As a main example, the final value problem of the heat equation on a smooth open set is
  treated, and non-zero Dirichlet data are shown to require a non-trivial extension of the
  compatibility condition by addition of an improper Bochner integral.  
\end{abstract}
\thanks{The second author is supported by the Danish 
Research Council, Natural Sciences grant no.~4181-00042.
\\[9\jot]{\tt Appeared online in Axioms, vol.\ 7 (2018), no.\ 31, 36 pp.; doi:10.3390/axioms7020031}}

\maketitle

\section{Introduction}\label{intro-sect}
\thispagestyle{empty}

In this article we establish well-posedness of final value problems
for a large class of parabolic differential equations. Seemingly, this 
clarifies a longstanding gap in the comprehension of such problems.

Taking the heat equation as a first example, we 
address the problem of characterising the functions $u(t,x)$ that, in a $C^\infty$-smooth bounded open set
$\Omega\subset\Rn$ with boundary $\partial\Omega$, 
fulfil the equations, where $\Delta=\partial_{x_1}^2+\dots+\partial_{x_n}^2$ denotes the Laplacian,
\begin{equation}  \label{heat-intro}
\left\{
\begin{aligned}
  \partial_tu(t,x)-\lap u(t,x)&=f(t,x) &&\quad\text{for $t\in\,]0,T[\,$,  $x\in\Omega$,}
\\
   u(t,x)&=g(t,x) &&\quad\text{for $t\in\,]0,T[\,$, $x\in\partial\Omega$,}
\\
  u(T,x)&=u_T(x) &&\quad\text{for $x\in\Omega$}.
\end{aligned}
\right.
\end{equation}

Motivation for doing so could be given by imagining a nuclear power plant, which is 
hit by a power failure at time $t=0$. Once power is regained at time $t=T$, and a 
measurement of the reactor temperature $u_T(x)$ is obtained, it is of course desirable to
calculate backwards in time to provide an answer to the question: were temperatures $u(t,x)$
around some earlier time $t_0<T$ high enough to cause a meltdown of the fuel rods ?

We provide here a theoretical analysis of such problems and prove that they are well-posed, i.e., they  have 
\emph{existence, uniqueness} and \emph{stability} of solutions $u\in X$ for given data
$(f,g,u_T)\in Y$, in~certain normed spaces $X$, $Y$ to be specified below.
The results were announced without proofs in the short note \cite{ChJo18}.

Although well-posedness is of decisive importance for the interpretation and accuracy of 
numerical schemes, which one would use in practice, such a theory has seemingly not been worked out
before. Explained roughly, our method is to provide a useful structure on the reachable set for a
general class of parabolic differential equations.

\subsection{Background}
Let us first describe the case $f=0$, $g=0$. Then the mere heat equation
$(\partial_t-\Delta)u=0$ is clearly solved for all $t\in\R$ by the function  
$u(t,x)=e^{(T-t)\lambda}v(x)$, if $v(x)$ is an eigenfunction of the Dirichlet realization
$\mlap_D$ of the Laplace operator with eigenvalue $\lambda$.

In view of this, the homogeneous final value problem
\eqref{heat-intro} would obviously have the above $u$ as a \emph{basic} solution if, coincidentally,
the final data $u_T(x)$ were given as the eigenfunction $v(x)$.
The theory below includes the set $\cal B$ of such basic solutions $u$ together with its linear hull
$\cal E=\Span\cal B$ and a certain completion $\overline{\cal E}$. 

It is easy to describe $\cal E$ in terms of the eigenvalues 
$0<\lambda_1\le\lambda_2\le\dots$ and the associated $L_2(\Omega)$-orthonormal
basis $e_1, e_2,\dots$ of eigenfunctions  of $\mlap_D$: corresponding to final data $u_T$ in $\op{span}(e_j)$, which are the $u_T$
having \emph{finite} expansions $u_T(x)=\sum_j \ip{u_T}{e_j} e_j(x)$ in $L_2(\Omega)$, the space 
$\cal E$ consists of solutions $u(t,x)$ being \emph{finite} sums
\begin{equation}
  \label{basic-id}
  u(t,x)=\textstyle{\sum_j}\, e^{(T-t)\lambda_j}\ip{u_T}{e_j}e_j(x).
\end{equation}
Moreover, at time $t=0$ there is, because of the finiteness, a vector $u(0,x)$ in $L_2(\Omega)$ that trivially fulfills
\begin{equation}
  \|u(0,\cdot)\|^2 =\textstyle{\sum_j}\, e^{2T\lambda_j}|\ip{u_T}{e_j}| ^2 <\infty.
  \label{basic-ineq}
\end{equation}

However, when summation is extended to all $j\in\N$, condition \eqref{basic-ineq} 
becomes very strong, as it is only satisfied for special $u_T$:
Weyl's law for the counting function, cf.~\cite[Ch.~6.4]{CuHi53}, entails the well-known
fact $\lambda_j= \cal O(j^{2/n})$, 
so a single term in \eqref{basic-ineq} yields $|\ip{u_T}{e_j}|\le c \exp({-Tj^{2/n}})$;
whence the $L_2$-coordinates of such $u_T$ decay rapidly for $j\to\infty$.

Condition \eqref{basic-ineq} has been known at least since the 1950s; the work of John~\cite{John55}
and Miranker \cite{Miranker61} are the earliest we know. While many authors have avoided an
analysis of it, Payne found it scientifically intolerable because $u_T$ is likely to be imprecisely 
measured; cf.\ his treatise \cite{Pay75} on the variety of methods applied to \eqref{heat-intro} until the mid 1970s. 

More recently e.g.\ Isakov~\cite{Isa98} emphasized the classical observation, found already in
\cite{Miranker61}, that \eqref{basic-id} implies a phenomenon of \emph{instability}. Indeed, the
sequence of final data $u_{T,k}=e_k$ has constant length $1$, yet via \eqref{basic-id} it gives the
initial states $u_k(0,x)=e^{T\lambda_k}e_k(x)$ having $L_2$-norms $e^{T\lambda_k}$, which clearly
blow up rapidly for $k\to\infty$. 

This $L_2$-instability cannot be explained away, of course, but it does \emph{not} rule out that
\eqref{heat-intro} is well-posed.  
It rather indicates that the $L_2$-norm is an insensitive choice for \eqref{heat-intro}.

In fact, here there is an analogy with the classical stationary Dirichlet problem 
\begin{equation} \label{Dir-pb}
  -\Delta u=f \quad\text{in $\Omega$},\qquad u=g \quad\text{on $\partial\Omega$}.
\end{equation}
This is \emph{unsolvable} for $u\in C^2(\overline{\Omega})$ given certain $f\in
C^0(\overline{\Omega})$, $g\in C^0(\partial\Omega)$: 
G{\"u}nther proved prior to 1934, cf.~\cite[p.85]{Gun67}, that when 
$f(x)= \chi(x)(3x_3^2|x|^{-2}-1)/\log|x|$ for some radial cut-off function $\chi\in
C_0^\infty(\Omega)$ equal to $1$ around the origin, $\Omega$ being the unit ball of $\R^3$, then 
$f\in C^0(\overline{\Omega})$ while the convolution
$w=\frac{1}{4\pi|x|}*f$ is in $C^1(\overline{\Omega})$ but \emph{not} in $C^2$ at $x=0$; 
so $w\in C^1(\overline\Omega)\setminus C^2(\overline\Omega)$.
Yet $w$ is the unique $C^1(\overline\Omega)$-solution of \eqref{Dir-pb} in the distribution space
${\cal D}'(\Omega)$ in the case $g$ is given as $g=w_{|\partial\Omega}$. 
Thus the $C^k$-scales constitute an insensitive choice for \eqref{Dir-pb}.
Nonetheless, replacing $C^2(\overline{\Omega})$ by its completion $H^1(\Omega)$ in the Sobolev norm
$(\sum_{|\alpha|\le1}\int_\Omega|D^\alpha u|^2\,dx)^{1/2}$, it is classical that 
\eqref{Dir-pb} is well-posed with $u$ in~$H^1(\Omega)$.  

To obtain similarly well-adapted spaces for \eqref{heat-intro} with $f=0$, $g=0$, one could
base the analysis \mbox{on \eqref{basic-ineq}}. Indeed, along with the above space
$\cal E$ of basic solutions, 
a norm $\vvvert u_T\vvvert$ on the  final data $u_T\in\op{span}(e_j)$ can be \emph{defined}
by \eqref{basic-ineq}, leading to the norm 
$\vvvert u_T \vvvert=(\sum_{j=1}^\infty  e^{2T\lambda_j}|\ip{u_T}{e_j}|^2)^{\frac12}$ on the $u_T$
that correspond to solutions $u$ in the completion $\overline{\cal E}$. This would give well-posedness of
\eqref{heat-intro} with $u\in\overline{\cal E}$;
cf.\ Remark~\ref{BE-rem}.

But the present paper goes much beyond this. For one thing, we have freed the discussion from
$-\Delta_D$ and its specific eigenvalue distribution by using sesqui-linear forms, cf.\ Lax--Milgram's lemma, 
which allowed us to extend the proofs to a general class of elliptic operators $A$.

Secondly we analyse the \emph{fully} inhomogeneous problem \eqref{heat-intro} 
for general $f$, $g$ in Section~\ref{heat-sect}. 
In this situation well-posedness is not just a matter of choosing the norm on the  data $(f,g,u_T)$
suitably (as~one might think from the above $\vvvert u_T\vvvert$). In fact, prior to this choice, 
one has to \emph{restrict} the $(f,g,u_T)$ to a subspace characterised by certain \emph{compatibility conditions}.
While such conditions are well known in the theory of parabolic boundary problems, they are shown here
to have a new and special form for final value problems.

Indeed, the compatibility conditions stem from the unbounded operator $u_T\mapsto u(0)$, which
maps the final data to the corresponding initial state in the presence of the source term $f$. 
The fact that this operator is well
defined, and that its domain endowed with the graph norm yields the data space, is the
leitmotif of this article.

\subsection{The Abstract Final Value Problem}
Let us outline our analysis given for a
Lax--Milgram operator $A$ defined in $H$ from a $V$-elliptic sesquilinear form $a(\cdot,\cdot)$ 
in a Gelfand triple, i.e.,
in a set-up of three Hilbert spaces $V\hookrightarrow H\hookrightarrow V^*$ 
having norms denoted $\|\cdot\|$, $|\cdot|$ and $\|\cdot\|_*$, and where $V$ is the form domain of $a$. 

In this framework, we consider the following general final value problem: given data  
\begin{equation}
  f\in L_2(0,T; V^*),\qquad u_T\in H,
\end{equation}
determine the $V$-valued vector distributions $u(t)$ on $\,]0,T[\,$, that is the $u\in{\cal D}'(0,T;V)$, fulfilling
\begin{equation}
  \label{eq:fvA-intro}
  \left\{
  \begin{aligned}
  \partial_tu +Au &=f  &&\quad \text{in ${\cal D}'(0,T;V^*)$},
\\
  u(T)&=u_T &&\quad\text{in $H$}.
\end{aligned}
\right.
\end{equation}
Classically a wealth of parabolic Cauchy problems with homogeneous boundary 
conditions have been efficiently treated with the triples $(H,V,a)$ and the $\cal D'(0,T;V^*)$
set-up in \eqref{eq:fvA-intro}. For this the reader may consult
the work of Lions and Magenes~\cite{LiMa72}, Tanabe~\cite{Tan79}, Temam~\cite{Tm}, Amann~\cite{Ama95}.
Also recently, e.g., Almog, Grebenkov, Helffer, Henry  studied
variants of the complex Airy operator via such \mbox{triples \cite{AlHe15,GrHelHen17,GrHe17}}, and our
results should at least extend to final value problems for those of their realisations that have
non-empty spectrum.

To compare \eqref{eq:fvA-intro} with the analogous Cauchy problem, we recall that whenever $u'+Au=f$
is solved under the initial condition $u(0)=u_0\in H$, for some $f\in L_2(0,T;V^*)$, there is a
unique solution $u$ in the Banach space 
\begin{equation}
  \begin{split}
  X=&L_2(0,T;V)\bigcap C([0,T];H) \bigcap H^1(0,T;V^*),
\\
  \|u\|_X=&\big(\int_0^T \|u(t)\|^2\,dt+\sup_{0\le t\le T}|u(t)|^2+\int_0^T (\|u(t)\|_*^2   +\|u'(t)\|_*^2)\,dt\big)^{1/2}.   
  \end{split}
  \label{eq:X-intro}
\end{equation}
For \eqref{eq:fvA-intro} it would thus be natural to envisage solutions $u$ in the same space $X$.
This turns out to be true, but only under substantial further conditions on the data $(f, u_T)$. 

To formulate these, we exploit that $-A$ generates an \emph{analytic} semigroup $e^{-tA}$ in $\B(H)$. 
This is crucial for the entire article, because analytic semigroups always are invertible in the
class of closed operators, as we show in Proposition~\ref{inj-prop}.
We denote its inverse by $e^{tA}$, consistent with the case $-A$ generates a group,
\begin{equation}
  \label{eq:inverse-intro}
  (e^{-tA})^{-1}=e^{tA}.
\end{equation}

Its domain is the Hilbert space $D(e^{tA})=R(e^{-tA})$ that is normed by $\|u\|=(|u|^2+|e^{tA}u|^2)^{1/2}$.
\mbox{In Proposition}~\ref{dom-prop}  we show that a non-empty spectrum, $\sigma(A)\ne\emptyset$,
yields strict inclusions
\begin{equation} \label{dom-intro}
  D(e^{t'A})\subsetneq D(e^{t A})\subsetneq H \qquad\text{ for  $0<t<t'$}.
\end{equation}

For $t=T$ these domains play a crucial role in the well-posedness result below, cf.~\eqref{eq:cc-intro},
where also the full yield $y_f$ of the source term $f$ on the system appears, namely
\begin{equation} \label{yf-eq}
  y_f= \int_0^T e^{-(T-s)A}f(s)\,ds.
\end{equation}
The map $f\mapsto y_f$ takes values in $H$, and it is a continuous surjection $y_f\colon
L_2(0,T;V^*)\to H$.

\begin{Theorem} \label{intro-thm}
  For the final value problem \eqref{eq:fvA-intro} to have a solution
  $u$ in the space $X$ in \eqref{eq:X-intro}, it is necessary and sufficient that the data
  $(f,u_T)$ belong to the subspace $Y$ of $L_2(0,T; V^*)\oplus H$ defined
  by the condition  
  \begin{equation}
    \label{eq:cc-intro}
    u_T-\int_0^T e^{-(T-t)A}f(t)\,dt \ \in\  D(e^{TA}).
  \end{equation}  
Moreover, in $X$ the solution $u$ is unique and depends continuously on the data $(f,u_T)$ in $Y$,
that is, we have $\|u\|_X\le c\|(f,u_T)\|_Y$,
when $Y$ is given the graph norm
\begin{equation}
  \label{eq:Y-intro}
  \|(f,u_T)\|_Y=
  \left(|u_T|^2+\int_0^T\|f(t)\|_*^2\,dt+\Big|e^{TA}\big(u_T-\int_0^Te^{-(T-t)A}f(t)\,dt\big)\Big|^2
  \right)^{1/2}.
\end{equation}
(The full statements are found in Theorem~\ref{fvp-thm} and Theorem~\ref{wp-thm} below.)
\end{Theorem}

Condition \eqref{eq:cc-intro} is a fundamental novelty for the above class of final value problems, but
more generally it also gives an important clarification for parabolic differential equations. 

As for its nature, we note that the data $(f,u_T)$ fulfilling \eqref{eq:cc-intro} form a 
Hilbert(-able) space $Y$ embedded into
$L_2(0,T; V^*)\oplus H$, in view of its norm in \eqref{eq:Y-intro}. 

Using the above $y_f$, \eqref{eq:Y-intro} is seen to be the graph norm of $(f,u_T)\mapsto e^{TA}(u_T-y_f)$,
which in terms of $\Phi(f,u_T)=u_T-y_f$ is the unbounded operator $e^{TA}\Phi$ from $L_2(0,T; V^*)\oplus H$ to $H$.  
As \eqref{eq:cc-intro} means that the operator $e^{TA}\Phi$ must be defined at
$(f,u_T)$, the space $Y$ is its domain.
Thus $e^{TA}\Phi$ is a key ingredient in the rigorous treatment of \eqref{eq:fvA-intro}.

The role of $e^{TA}\Phi$ is easy to elucidate in control theoretic terms: 
its value $e^{TA}\Phi(f,u_T)$ simply equals the particular initial state $u(0)$ that is steered by $f$ 
to the final state $u(T)=u_T$ at time $T$;
cf.~\eqref{eq:bijection-intro} below.

Because of $e^{-(T-t)A}$ and the integral over $[0,T]$, 
\eqref{eq:cc-intro} involves \emph{non-local} operators in both space and time as an inconvenient
aspect\,---\,exacerbated by use of the abstract domain $D(e^{TA})$, which for longer lengths $T$ of 
the time interval gives increasingly stricter conditions; cf.\ \eqref{dom-intro}. 

Anyhow, we propose to regard \eqref{eq:cc-intro} as a \emph{compatibility} condition on
the data $(f,u_T)$, and thus we generalise the notion of compatibility.
 
For comparison we recall that Grubb and Solonnikov~\cite{GrSo90} made a systematic investigation
of a large class of \emph{initial}-boundary problems of parabolic (pseudo-)differential equations
and worked out compatibility conditions, which are necessary and sufficient for
well-posedness in full scales of anisotropic $L_2$-Sobolev spaces. Their conditions are explicit
and local at the curved corner $\partial\Omega\times\{0\}$, except for 
half-integer values of the smoothness $s$ that were shown to require so-called coincidence, which 
is expressed in integrals over the product of the two boundaries $\{0\}\times\Omega$ and
$\,]0,T[\,\times\,\partial\Omega$; hence it also is a non-local condition.  

However, whilst the conditions of Grubb and Solonnikov~\cite{GrSo90} are decisive for the solution's regularity, 
condition \eqref{eq:cc-intro} is crucial for the \emph{existence} question; cf.\ the theorem.

Previously, uniqueness was shown by Amann~\cite[Sect.~V.2.5.2]{Ama95} in a $t$-dependent set-up, 
but injectivity of $u(0)\mapsto u(T)$ was proved much earlier for problems with $t$-dependent
sesquilinear forms by Lions and Malgrange~\cite{LiMl60}.

Showalter~\cite{Sho74} attempted to characterise the possible $u_T$ in terms of Yosida
approximations for $f=0$ and $A$ having half-angle $\frac\pi4$.
As an ingredient, 
invertibility of analytic semigroups was claimed in \cite{Sho74} for such $A$, but the proof was flawed as
$A$ can have semi-angle $\pi/4$ even if $A^2$ is not accretive; cf.\ our example in Remark~\ref{Show-rem}.

Theorem~\ref{intro-thm} is proved largely by comparing with the corresponding problem $u'+Au=f$,
$u(0)=u_0$. It is well known in functional analysis, cf.\ \eqref{eq:X-intro}, that this is
well-posed for data $f\in L_2(0,T;V^*)$, $u_0\in H$, with solutions $u\in X$. 
However, as shown below by adaptation of a classical argument, $u$ is also in this
set-up necessarily given by Duhamel's principle, or the 
variation of constants formula, for the analytic semigroup $e^{-tA}$ in $V^*$, 
\begin{equation}
  \label{eq:bijection-intro}
  u(t)=e^{-tA}u(0)+ \int_0^t e^{-(t-s)A}f(s)\,ds.
\end{equation}

For $t=T$ this yields a \emph{bijective correspondence} $u(0)\leftrightarrow u(T)$ between the initial
and terminal states (in particular backwards uniqueness of the solutions in the large class $X$)---but
this relies crucially on the previously mentioned invertibility of  $e^{-tA}$; cf.\
\eqref{eq:inverse-intro}.

As a consequence of \eqref{eq:bijection-intro} one finds the necessity of \eqref{eq:cc-intro}, as the
difference $\Phi(f,u_T)=u_T-y_f$ in \eqref{eq:cc-intro} must equal the vector $e^{-TA}u(0)$, 
which obviously belongs to $D(e^{TA})$.  

Moreover, \eqref{eq:bijection-intro} yields that $u(T)$ in a natural way consists of two 
parts, that differ radically even when $A$ has nice properties: 

First, $e^{-tA}u(0)$ solves the semi-homogeneous problem with $f=0$, and for $u(0)\ne0$ there is 
the precise property in non-selfadjoint dynamics that the ``height'' function $h(t)$ is
\emph{strictly convex}, 
\begin{equation}
  h(t)= |e^{-tA}u(0)|.
\end{equation}
This is shown in  Proposition~\ref{sc-prop} when $A$ belongs to the broad class of hyponormal operators, 
studied by Janas \cite{Jan94}, 
or in case $A^2$ is accretive; then $h(t)$ is also strictly decreasing with $h'(0)\le-m(A)$, where
$m(A)$ is the lower bound of $A$.

The stiffness inherent in strict convexity is supplemented by the fact that $u(T)=e^{-TA}u(0)$ is
confined to a dense, but very small space, as by a well-known property of analytic semigroups, 
\begin{equation}
  \label{DAn-cnd}
  u(T)\in \textstyle{\bigcap_{n\in\N}} D(A^n).
\end{equation}

Secondly, for $u_0=0$ the integral in   \eqref{eq:bijection-intro} solves the initial value problem, and it has a
rather different nature since its final value $y_f$ in \eqref{yf-eq} is surjective
$y_f\colon L_2(0,T;V^*)\to H$, hence 
can be \emph{anywhere} in $H$, regardless of the Lax--Milgram operator $A$ in our set-up.
This we show in Proposition~\ref{Ryf-prop} using a kind of control-theoretic argument in case $A$ is self-adjoint with compact
inverse; and for general $A$ by means of the Closed Range Theorem,
cf.~Proposition~\ref{yf-prop}.

For the reachable set of the equation $u'+Au=f$, or rather the possible final data
$u_T$, they will be a sum of an arbitrary vector $y_f$ in $H$ and a
term $e^{-TA}u(0)$ of great stiffness (cf.~\eqref{DAn-cnd}). Thus $u_T$ can be prescribed in the affine space
$y_f+D(e^{TA})$. As any $y_f\ne0$ will push the dense set $D(e^{TA})\subset H$
in some arbitrary direction, $u(T)$ can be expected \emph{anywhere} in $H$ (unless $y_f\in D(e^{TA})$ is known a priori).
Consequently neither $u(T)\in D(e^{TA})$ nor \eqref{DAn-cnd} can be expected to hold for $y_f\ne0$, 
not even if its norm $|y_f|$ is much smaller than $|e^{-TA}u(0)|$.

As for final state measurements in real life applications, we would like to prevent a
misunderstanding by noting that it is only under the peculiar circumstance that $y_f=0$ is known
\emph{a priori} to be an \emph{exact} identity that \eqref{DAn-cnd} would be a valid expectation on $u(T)$.

Indeed, even if $f$ is so small that it is (quantitatively)
insignificant for the time development of the system governed by $u'+Au=f$, so that $f=0$ is a
valid dynamical approximation, the (qualitative) \emph{mathematical} expectation that $u(T)$
should fulfill \eqref{DAn-cnd} cannot be justified from such an approximation; cf.\ the
above. 

In view of this fundamental difference between the problems that are truly and merely approximately homogeneous, 
it seems that proper understanding of final value problems is facilitated by treating 
inhomogeneous problems from the very beginning. 

\subsection{The Inhomogeneous Heat Problem}
For \eqref{heat-intro} with general data $(f,g,u_T)$ the above is applied with
$A=\mlap_D$, that is the Dirichlet realisation of the Laplacian. The results are
analogous, but less simple to state and more demanding to obtain.

First of all, even though it is a linear problem,  
the compatibility condition \eqref{eq:cc-intro}
\emph{destroys} the old trick of reducing to boundary data $g=0$, 
for when $w\in H^1$ fulfils $w=g\ne0$ on the curved boundary $S=\,]0,T[\,\times\partial\Omega$, then $w$
\emph{lacks} the regularity needed to test  \eqref{eq:cc-intro} on the data 
$(\tilde f,0,\tilde u_T)$ of the reduced problem; cf.\ \eqref{tilde-data} ff.

Secondly, it is, therefore, non-trivial to clarify that every $g\ne0$ \emph{does} give rise to an extra term $z_g$, 
\mbox{in the }sense that \eqref{eq:cc-intro} is replaced by the compatibility condition
\begin{equation}\label{eq:yzcc-intro}
  u_T-y_f+z_g\in D(e^{-T\Delta_{D}}).  
\end{equation}

Thirdly, due to the low reqularity, it requires technical diligence to show that
$z_g$, despite the singularity of $\Delta e^{(T-s) \Delta_{D}}$ at $s=T$,  has the structure
of a single convergent improper Bochner integral, namely 
\begin{equation}
  \label{eq:zg-intro}
  z_g = \dashint_0^T \Delta e^{(T-s) \Delta_D} K_0 g(s) \,ds.  
\end{equation}
The reader is referred to Section~\ref{heat-sect} for the choice of the Poisson operator
$K_0$ and for an account of the results on the fully inhomogeneous problem in
\eqref{heat-intro},
especially Theorem~\ref{yz-thm} and Corollary~\ref{yz-cor}, which we sum up here:

\begin{Theorem}  \label{yz'-thm}
For given data $f \in L_2(0,T; H^{-1}(\Omega))$, $g \in H^{1/2}(S)$,  
$u_T \in L_2(\Omega)$ 
the final value problem \eqref{heat-intro} is solved by a function $u$ in
$X_1= L_2(0,T; H^1(\Omega)) \bigcap C([0,T];L_2(\Omega)) \bigcap H^1(0,T; H^{-1}(\Omega))$,
if and only if the data in terms of \eqref{yf-eq} and \eqref{eq:zg-intro} satisfy the compatibility
condition \eqref{eq:yzcc-intro}.
In the affirmative case, $u$ is uniquely determined in $X_1$ and has the representation, with all
{terms in $X_1$,}
\begin{align} 
  u(t) = 
   e^{t \Delta_{D}} e^{-T \Delta_{D}} (u_T - y_f + z_g) 
   + \int_0^t e^{(t-s)\Delta} f(s) \,ds - \dashint_0^t \Delta e^{(t-s) \Delta_{D}} K_0 g(s) \,ds,
\end{align}
The unique solution $u$ in $X_1$ 
depends continuously on the data $(f,g,u_T)$ in the Hilbert space $Y_1$, when these are given 
 the norms in \eqref{Xheat-nrm} and \eqref{Yheat-nrm} below, respectively.
\end{Theorem}

\subsection{Contents}
Our presentation is aimed at describing methods and consequences in a concise way,
readable for a broad audience within evolution problems.
Therefore we have preferred a simple set-up, leaving many examples and extensions to future work,
cf.\ Section~\ref{final-sect}.

Notation is given in Section~\ref{prel-sect} together with 
the set-up for Lax--Milgram operators and semigroup theory. 
Some facts on forward evolution problems are recalled in Section~\ref{ivp-sect},
followed by our analysis of abstract final value problems in Section~\ref{fvp-sect}. 
The heat equation and its final and boundary
value problems are treated in Section~\ref{heat-sect}. Section~\ref{final-sect}
concludes with remarks on the method's applicability and notes on the literature of the problem.

\section{Preliminaries}\label{prel-sect}
In the sequel specific constants will appear as $C_j$, $j\in \N$, 
whereas constants denoted by $c$ may vary from place to place. $\1_S$ denotes the characteristic
function of the set $S$. 

Throughout $V$ and $H$ denote two separable Hilbert spaces, such that $V$
is algebraically, topologically and densely contained in $H$. Then
there is a similar inclusion into the anti-dual $V^*$, i.e.\ the space of conjugated linear
functionals on $V$,
\begin{align}  \label{VHV-eq}
  V \subseteq H \equiv H^* \subseteq V^*.
\end{align}
$(V,H,V^*)$ is also known as a Gelfand triple. 
Denoting the norms by $\|\cdot\|$, $|\cdot|$ and $\|\cdot\|_{*}$, respectively, there are constants
such that for all $v \in V$,  
\begin{equation}\label{eq:VV*emb}
  \|v\|_{*} \leq C_1 |v| \leq C_2 \|v\|.
\end{equation}
The inner product on $H$ is denoted by $\ip{\cdot}{\cdot}$; and
the sesquilinear scalar product on $V^*\times V$ by
$\dualp{\cdot}{\cdot}_{V^*,V}$ or $\dualp{\cdot}{\cdot}$, it fulfils
$|\dualp{u}{v} | \leq \| u \|_{*} \| v \|$.
The second inclusion in \eqref{VHV-eq} means that for $u \in H$,
\begin{equation}
  \label{eq:HV*}
  \dualp{u}{v} = \ip{u}{v} \quad\text{ for all $v \in V$}.
\end{equation}

For a linear transformation $A$ in $H$, the domain is written $D(A)$, while $R(A)$
denotes its range and $Z(A)$ its null-space. $\rho(A)$, $\sigma(A)$ and 
$\nu(A)=\{\,\ip{Au}{u}\mid u\in D(A),\ |u|=1\,\}$ denote the resolvent set, 
spectrum and numerical range, while 
$m(A)=\inf\Re\nu(A)$ is the lower bound of $A$.
Throughout $\mathbb B(H)$ stands for the Banach space of bounded linear operators on $H$.
  
For a given Banach space $B$ and $T>0$, we denote by $L_1(0,T;B)$ the space of equivalence classes
of functions $f\colon [0,T]\to B$ that are 
strongly measurable with $\int_0^T\| f(t)\|\,dt$ finite. For such $f$ the Bochner integral is
denoted by $\int_0^T f(t)\,dt$, cf.\ \cite{Yos80};  it fulfils 
$\dual{\int_0^T f(t)\,dt}{\lambda}=\int_0^T \dual{f(t)}{\lambda}\,dt$
for every functional $\lambda$ in the dual
space $B'$. Likewise
$L_2(0,T,B)$ consists of the strongly measurable $f$ with finite norm $(\int_0^T\|f(t)\|^2\,dt)^{1/2}$.

On an open set $\Omega\subset\Rn$, $n\ge1$, the space $C^\infty_0(\Omega)$ consists of the
infinitely differentiable functions having compact support in $\Omega$; it is given the usual $\cal
L\cal F$-topology, cf.\ \cite{G09,Swz66}. The dual space of
continuous linear functionals $\cal D'(\Omega)$ is the distribution space on $\Omega$. We use the
standard distribution theory as exposed by Grubb \cite{G09} and H\"ormander \cite{H}.
 
More generally, the space of $B$-valued vector distributions is denoted by $\cal D'(\Omega;B)$; it
consists of the continuous linear maps $\Lambda\colon C^\infty_0(\Omega)\to B$,
 cf.~\cite{Swz66}, the value of which
at $\varphi\in C^\infty_0(\Omega)$ is indicated by $\dual{\Lambda}{\varphi}$. If $\Omega$ is the
interval $\,]0,T[\,$ we also write $\cal D'(\Omega;B)=\cal D'(0,T;B)$.

The Sobolev space $H^1(0,T;B)$ consists of the $u\in\cal D'(0,T;B)$ for which both $u$, $u'$ belong
to $L_2(0,T;B)$; it is normed by $\int_0^T(\|u\|^2 +\|u'\|^2)\,dt)^{1/2}$. 
More generally $W^{1,1}(0,T;B)$ is defined by replacing $L_2$ by $L_1$.

\subsection{Lax--Milgram Operators} \label{LaxM-ssect}
Our main tool will be the Lax--Milgram operator associated to an elliptic sesquilinear
form, cf.\ the set-up in \cite[Sect.~12.4]{G09}. For the reader's sake we review this, 
also to establish a few additional points from the proofs in \cite{G09}.

We let $a(\cdot, \cdot)$ be a bounded, $V$-elliptic sesquilinear form on $V$, 
i.e., certain $C_3, C_4>0$ fulfil for all $u,v \in V$ 
\begin{align}
  |a(u,v)|  \leq C_3 \| u\| \| v\|, \qquad
  \Re a(v,v)  \geq C_4 \| v \|^2.\label{eq:V-ell}
\end{align}

Obviously, the adjoint sesquilinear form $a^*(u,v)=\overline{a(v,u)}$ inherits these properties 
(with the same $C_3$, $C_4$), and so does the ``real part'',
$a_{\Re}(u,v)=\frac12(a(u,v)+a^*(u,v))$.
Since $a_{\Re}(u,u)\ge0$, the form $a_{\Re}$ is an inner product on $V$, inducing the equivalent norm
\begin{equation}
  \label{eq:3norm}
  \vvvert u \vvvert=a_{\Re}(u,u)^{1/2},\quad\text{for $u\in V$}.
\end{equation}

We recall that $s(u,v) =\ip{Su}{v}_V$ gives a bijective correspondence between bounded
sesquilinear forms $s(\cdot,\cdot)$ on $V$ and bounded operators $S\in\B(V)$, which is
isometric since $\|S\|$ equals  the operator norm 
of the sesquilinear form $|s| = \sup\left\{|s(u,v)| \bigm| \|u\| = 1 = \|v\| \right\}$.
So the given form $a$ induces an $\A_0 \in \B(V)$ given by 
\begin{align}\label{eq:A0}
  a(u,v) = \ip{\A_0 u}{v}_V \quad \forall u,v \in V;
\end{align}
and the adjoint form $a^*$ similarly induces an operator $\A_0^*\in \B(V)$, which is seen at once
to be the adjoint of $\A_0$ in the sense that $\ip{\A_0^*v}{u}_V=\ip{v}{\A_0u}_V$.
  
The $V$-ellipticity in \eqref{eq:V-ell} shows that $\A_0$, $\A_0^*$ are both injective
with positive lower bounds $m(\A_0),\,m(\A_0^*)\ge C_4$, 
so $\A_0$, $\A_0^*$ are in fact bijections on $V$ (cf. \cite[Theorem 12.9]{G09}).

By Riesz's representation theorem, there exists a bijective
isometry $J \in \B(V,V^*)$ such that for every $v^* = J \tilde{v}$ one has 
$\dualp{J \tilde{v}}{v} = \ip{\tilde{v}}{v}_V$ for all $v \in V$.
Therefore $\A := J \circ \A_0$ is an operator in $\B(V,V^*)$, for which \eqref{eq:A0} gives
\begin{align}  \label{eq:calA}
  \dualp{\A u}{v} = a(u,v), \quad \forall u,v \in V.
\end{align}
Similarly $\A ' := J \circ \A_0^*$ fulfils
$\dualp{\A' u}{v}=\ip{\A_0^*u}{v}_V = a^*(u,v)$ for all $u,v \in V$.

Clearly $\A$ and $\A'$ are bijections, as composites of such.
Hence they give rise to a Hilbert space structure on $V^*$ with the inner product
\begin{align}
  \ip{w_1}{w_2}_{V^*} = a_{\Re}(\A^{-1} w_1, \A^{-1}w_2),
\label{ipV*-id}
\end{align}
inducing the norm $\vvvert w \vvvert_{*} = a_{\Re}(\A^{-1} w, \A^{-1}w)^{1/2}=\vvvert
\A^{-1}w\vvvert $ on $V^*$, equivalent to $\|w \|_{*}$.

The Lax--Milgram operator $A$ is defined by restriction of $\A$ to an operator in $H$, i.e.,
\begin{equation}
  A v = \A v \quad \text{for } v \in D(A):=\A^{-1}(H).
\end{equation}
So $D(A)$ consists of the $u\in V$ for which some $f\in H$ fulfils $\ip{f}{v}=a(u,v)$ for all
$v\in V$.

The reader may consult \cite[Sect.~12.4]{G09} for elementary proofs of the following: 
$A$ is closed in $H$, with $D(A)$ dense in $H$ as well as in $V$;
in $H$ also $\A'$ has these properties, and 
it equals the adjoint of $A$ in $H$; i.e., $A'|_{\A'^{-1}(H)}=A^*$.
As $A$ is closed, $D(A)$ is a {Hilbert} space with the graph norm
$\|v\|_{D(A)}^2 = |v|^2 + |Av|^2$,
and $D(A)\hookrightarrow V$ is bounded due to \eqref{eq:V-ell}.
Geometrically, $\sigma(A)$ and $\nu(A)$ are contained in the sector of $z\in\C$ given by
\begin{align} \label{Asect-id}
  |\Im z|\leq C_3 C_4^{-1} \Re z.
\end{align}
Actually $0 \in \rho(A)$ since $a$ is $V$-elliptic, so $A^{-1} \in \B(H)$; moreover
$m(A)\ge C_1C_4/C_2>0$.

Both the closed operator $A$ in $H$ and its extension $\A\in\B(V,V^*)$ are used throughout.
(For simplicity, they were both denoted by $A$ in the introduction, though.)

\subsection{The Self-Adjoint Case}
  \label{sa-ssect}
As is well known, if $A$ is selfadjoint, i.e. $A^*=A$ (or $a^*=a$), and has \emph{compact} inverse,  then
$H$ has an orthonormal basis of eigenvectors of $A$, which can be scaled to orthonormal bases 
of $V$ and $V^*$. This is recalled, because our results can be
given a more explicit form in this case, e.g.\ for $-\Delta$ in \eqref{heat-intro}.

The properties that $A$ is selfadjoint, closed, and densely defined with dense range in $H$ carry over
to $A^{-1}$, e.g.~\cite[Thm.~12.7]{G09}, so when $A^{-1}$ in addition is compact in $H$ (e.g., if
$V\hookrightarrow H$ is compact),
then the spectral theorem for compact selfadjoint operators states that 
$H$ has an orthonormal basis $(e_j)$ consisting of eigenvectors of $A^{-1}$,  
where the eigenvalues $\mu_j$ of $A^{-1}$ by the positivity can be ordered such that
\begin{align}
\mu_1 \geq \mu_2 \geq \ldots \geq \mu_j \ge\dots > 0, \quad \text{with $\mu_j \to 0$ if $j\to\infty$}. 
\end{align}

The orthonormal basis $(e_j)$ also consists of eigenvectors of $A$ with
eigenvalues  $\lambda_j={1}/{\mu_j}$. Hence
$\sigma(A) = \sigma_{\point}(A) = \{\, \lambda_j\mid j\in\N \,\}$. 
Indeed, $\sigma_{\operatorname{res}}(A)=\emptyset$ as $A^*=A$; and $A^{-1}\in\B(H)$ while
$A - \nu I = (\nu^{-1}I - A^{-1}) \nu A$ has a bounded inverse
for $\nu \neq \lambda_j$, as $\nu^{-1} \notin \sigma(A^{-1})$.

As $a_{\Re}=a$ here, $V$ is now renormed by $\vvvert v \vvvert^2 = a(v,v)$. However, if moreover 
$V$ is considered with $a(u,v)$ as inner product, then $\A\colon V\to V^*$ is the Riesz isometry; and one has 

\begin{fact}\label{fact:1}
For every $v \in V$ the $H$-expansion $v=\sum_{j=1}^{\infty} \ip{v}{e_j} e_j$ converges in $V$.
Moreover, the sequence $(e_j/\sqrt{\lambda_j})_{j\in\N}$ is an orthonormal basis for $V$, and
$\vvvert v \vvvert^2 = \sum_{j=1}^\infty \lambda_j| \ip{v}{e_j}|^2 $. 
\end{fact}

\begin{proof}
The $e_j/\sqrt{\lambda_j}$ are orthonormal in $V$ since 
$a(e_j,e_k) = \dualp{\A e_j}{e_k} = \lambda_j \ip{e_j}{e_k}$, cf. \eqref{eq:calA}.
They also yield a basis for $V$ since similarly $w\in V\ominus\Span(e_j/\sqrt{\lambda_j})$ implies $w=0$.
As $\lambda_j > 0$, expansion of any $v$ in $V$ gives 
\begin{align}
v & = \sum_{j=1}^{\infty} a(v, \lambda_j^{-1/2} e_j ) \lambda_j^{-1/2} e_j = \sum_{j=1}^{\infty}
\overline{a(e_j,v)} \lambda_j^{-1}  e_j = \sum_{j=1}^{\infty} \ip{v}{e_j} e_j, 
\end{align}
whence the rightmost side converges in $V$.
This means that $v = \sum_{j=1}^{\infty} \sqrt{\lambda_j}\ip{v}{e_j} e_j/\sqrt{\lambda_j}$
is an orthogonal expansion in $V$, whence $\vvvert v \vvvert^2$ has the stated expression.
\end{proof}

For $V^*$ the set-up \eqref{ipV*-id}, \eqref{eq:calA} here gives
$\ip{w_1}{w_2}_{V^*}=a(\A^{-1}w_1,\A^{-1}w_2)=\dualp{w_1}{\A^{-1}w_2}$.

\begin{fact}  \label{fact:2}
For every $w \in V^*$ the expansion $w=\sum_{j=1}^{\infty} \dualp{w}{e_j} e_j$ converges in
$V^*$. Moreover, the sequence $(\sqrt{\lambda_j} e_j)_{j\in\N}$ is an orthonormal basis of $V^*$ and
$\vvvert w \vvvert_{*}^2 = \sum_{j=1}^{\infty} \lambda_j^{-1} | \dualp{w}{e_j}|^2$. 
\end{fact}

\begin{proof}
$(\sqrt{\lambda_j} e_j)$ is orthonormal as
$\ip{e_j}{e_k}_{V^*} = \dualp{e_j}{\A^{-1}e_k} = \lambda_k^{-1} \ip{e_j}{e_k}$;
and if $w \in V^*$ {for all $j$} fulfils $0 = \dualp{e_j}{\A^{-1} w} =\ip{e_j}{\A^{-1} w}$ , 
then $w = 0$ as $\A^{-1}$ is injective. Therefore  
\begin{align}
   w = \sum_{j=1}^{\infty} \ip{w}{\lambda_j^{1/2} e_j}_{V^*} \lambda_j^{1/2} e_j 
   = \sum_{j=1}^{\infty} \dualp{w}{\A^{-1} e_j} \lambda_j e_j =  \sum_{j=1}^{\infty} \dualp{w}{e_j} e_j,
\end{align}
so the rightmost side converges in $V^*$, and the expression for $\vvvert w \vvvert_{*}^2$ results.
\end{proof}

\subsection{Semigroups}\label{sec:semigroup_theory}
Assuming that the reader is familiar with the theory of semigroups $e^{t\Gen}$, we review a few needed facts in a
setting with a general complex Banach space $B$. 
The books of Pazy~\cite{Paz83}, Tanabe~\cite{Tan79} and Yosida~\cite{Yos80} may serve as general references.

The generator is
$\Gen x=\lim_{t\to0^+}\frac1t(e^{t\Gen}x-x)$, with domain $D(\Gen)$ consisting of the $x\in B$ for which the limit exists.
$\Gen$ is a densely defined, closed linear
operator in $B$ that for certain $\omega \geq 0$ and $M \geq 1$ satisfies 
$\|(\Gen-\lambda)^{-n}\|_{\B(B)}\le M/(\lambda-\omega)^n$ for $\lambda>\omega$, $n\in\N$.

The corresponding semigroup of operators is written $e^{t\Gen}$, it belongs to $\B(B)$ with
\begin{align}  \label{Mo-ineq}
  \|e^{t\Gen}\|_{\B(B)} \leq M e^{\omega t} \quad \text{ for } 0 \leq t < \infty.
\end{align}
Its basic properties are that $e^{t\Gen}e^{s \Gen}=e^{(s+t)\Gen}$ for $s,t\ge0$, $e^{0\Gen}=I$,
$\lim_{t\to0^+}e^{t \Gen}x=x$ for $x\in B$, and the first of these gives at once the range
inclusions 
\begin{equation}
  \label{incl-eq}
  R(e^{(s+t)\Gen})\subset R(e^{t \Gen})\subset B.  
\end{equation}

The following well-known theorem gives a criterion for $\Gen$ to generate an analytic semigroup that is
uniformly bounded, i.e.,  has $\omega=0$. It summarises the most relevant parts of 
Theorems~1.7.7 and 2.5.2 in \cite{Paz83}, and it involves sectors of the form
\begin{equation}\label{eq:Sigma}
  \Sigma :=\left\{ \lambda \in\C \Bigm| |\arg\lambda | < \frac{\pi}{2} + \theta \right\} \cup \left\{ 0 \right\}.
\end{equation}

\begin{Theorem}\label{Pazy-thm}
{If} $\theta \in\,]0,\frac{\pi}{2}[\,$ and $M>0$ are such that the resolvent set $\rho(\Gen) \supseteq
\Sigma$ and
\begin{equation} \label{lIA-est}
\|(\lambda I - \Gen)^{-1} \|_{\B(B)} \leq \frac{M}{|\lambda|}, \quad \text{ for $\lambda\in\Sigma$, $\lambda \neq 0$},
\end{equation}
then $\Gen$ generates an analytic semigroup $e^{z \Gen}$ for
$|\arg z | < \theta$, for which  $\|e^{z\Gen}\|$ is bounded for $|\arg z|\le\theta'<\theta$,
and $e^{t \Gen}$ is differentiable in 
$\B(B)$ for $t>0$ with $(e^{t\Gen})' = \Gen e^{t\Gen}$. Here
\begin{align}\label{AetA-est}
\|\Gen e^{t\Gen}\|_{\B(B)} \leq \frac{c}{t} \quad \text{for } t>0.
\end{align}
\end{Theorem}
Furthermore, if $e^{t\Gen}$ is analytic, $u'=\Gen u$, $u(0)=u_0$ is uniquely solved by $u(t)=e^{t\Gen}u_0$ for \emph{every} $u_0\in B$.

\subsubsection{Injectivity}
Often it is crucial to know whether the semigroup $e^{t\Gen}$ consists of \emph{injective} operators.
Injectivity is e.g.\ equivalent to the geometric property that the trajectories of two solutions
$e^{t\Gen}v_0$ and $e^{t\Gen}w_0$ of $u'=\Gen u$  have no point of confluence in $B$ for $v_0\ne
w_0$.

However, the literature seems to have focused on examples with non-invertibility of $e^{t\Gen}$, \mbox{e.g., 
\cite[Ex.~2.2.1]{Paz83}.}
But injectivity always holds in the analytic case, as we now show:

\begin{Proposition}  \label{inj-prop}
When a semigroup $e^{z\Gen}$ on a complex Banach space $B$ is analytic $S\to \B(B)$ in the sector 
$S= \left\{z \in \C \mid | \arg z| < \theta \right\}$ for some
$\theta>0$, then $e^{z\Gen}$ is \emph{injective} for all $z \in S$.
\end{Proposition}

\begin{proof}
Let $e^{z_0 \Gen} u_0 = 0$ hold for some $u_0 \in B$, $z_0 \in S$.
The analyticity of $e^{z\Gen}$ in $S$ carries over to the map
$f\colon z \mapsto e^{z\Gen}u_0$, and to $g_v\colon z \mapsto \dual{v}{f(z)}$ for arbitrary
$v$ in the dual space $B'$. 
So $g_v$ has in a ball $B(z_0,r)\subset S$ the Taylor expansion 
\begin{align}
  g_v(z) = \sum_{n=0}^{\infty} \frac{1}{n!} \dual{v}{f^{(n)}(z_0)}(z-z_0)^n.
\end{align}
By the properties of analytic semigroups (cf.~\cite[Lem.~2.4.2]{Paz83}) 
and of $u_0$,
\begin{align}
  f^{(n)}(z_0) = \Gen^n e^{z_0 \Gen}u_0 = 0 \quad \text{ for all } n\geq 0, 
\end{align}
so that $g_v \equiv 0$ holds on $B(z_0,r)$ and consequently on $S$ by unique analytic extension.

Now $f(z_1) \neq 0$ would yield $g_v(z_1) \ne 0$ for a suitable $v$ in $B'$,
hence $f \equiv 0$ on $S$ and
\begin{align}
u_0 = \lim_{t \rightarrow 0} e^{t\Gen} u_0 = \lim_{t \rightarrow 0}f(t) = 0,
\end{align}
since $e^{t\Gen}$ is a strongly continuous semigroup. Altogether $Z(e^{z_0\Gen})=\{0\}$ is proved.
\end{proof}

\begin{Remark} \label{Showinj-rem}
We have only been able to track a claim of the injectivity in Proposition~\ref{inj-prop} in case
$z>0$, $\theta\le \pi/4$ and $B$ is a Hilbert space; cf.\ Showalter's paper \cite{Sho74}. 
But his proof is flawed, as $\Gen^2$ is non-accretive for some $\Gen$ with $\theta\le \pi/4$, cf.\
the counter-example in Remark~\ref{Show-rem} below. 
\end{Remark}

\begin{Remark}\label{rem:bijektiv_onb}
Injectivity also follows directly when $\Gen$ is defined on a Hilbert space $H$ having
an orthonormal basis $(e_n)_{n\in\N}$ such that $\Gen e_j = \lambda_j e_j$:
Clearly $e^{t\Gen}e_j= e^{t \lambda_j}e_j$ as both sides satisfy
$x' -\Gen x = 0$, $x(0) = e_j$.
So if $e^{t\Gen} v = 0$,  boundedness of $e^{t \Gen}$ gives
\begin{align}
   0 = e^{t\Gen} v  = \sum \ip{v}{e_j} e^{t \Gen} e_j  = \sum \ip{v}{e_j} e^{t\lambda_j} e_j,
\end{align}
so that $v\perp e_j$ for all $j$, and thus $v\in\Span(e_n)^\perp=H^\perp=\{0\}$. Hence $e^{t\Gen}$ is invertible for such $\Gen$.
\end{Remark}

We have chosen to use the symbol $e^{-t\Gen}$ to denote the inverse of the \emph{analytic} semigroup 
$e^{t\Gen}$ generated by $\Gen$, consistent with the case in which $e^{t\Gen}$ does form a group in
$\mathbb{B}(B)$, i.e., 
\begin{equation} \label{eq:inverse}
  e^{-t\Gen} := (e^{t\Gen})^{-1} \qquad\text{for all $t\in\R$}.
\end{equation} 
This notation is convenient for our purposes (with some diligence).

For simplicity we observe the following when $B=H$ is a Hilbert space and $t>0$:
clearly $e^{-t\Gen}$ maps $D(e^{-t\Gen})= R(e^{t\Gen})$ bijectively onto $H$, and
it is an unbounded closed operator in $H$.
As $(e^{t\Gen})^*=e^{t\Gen^*}$ also is analytic, so that $Z(e^{t\Gen^*})=\{0\}$ by Proposition~\ref{inj-prop}, we have
$\overline{D(e^{-t\Gen})}=H$, i.e., the domain is dense in $H$.

A partial group phenomenon and other algebraic properties are collected here:

\begin{Proposition}  \label{inverse-prop}
The inverses $e^{-t\Gen}$ in \eqref{eq:inverse} form a semigroup of unbounded operators,
\begin{equation}  \label{sg-id}
  e^{-t\Gen}e^{-s\Gen}= e^{-(t+s)\Gen} \qquad \text{for $t, s\ge0$}.
\end{equation}
This extends to {$(s,t)\in\,]-\infty,0]\times \R$,} but the right-hand side may be unbounded for $t+s>0$. 

Moreover, as unbounded operators the $e^{-t\Gen}$ commute with $e^{s \Gen}\in \B(H)$, i.e.,
\begin{equation}
  \label{comm-eq}
  e^{s \Gen}e^{-t\Gen}\subset e^{-t\Gen}e^{s\Gen} \qquad \text{for $t,s\ge0$},
\end{equation}
and there is a descending chain of domain inclusions
\begin{align}  \label{eq:domain_inclusion}
  D(e^{-t'\Gen}) \subset D(e^{-t\Gen}) \subset H \qquad\text{for $0<t<t'$}.
\end{align} 
\end{Proposition}

\begin{proof}
When $s,t\ge0$, clearly $e^{-t\Gen}e^{-s\Gen}e^{(s+t)\Gen}=I_H$
holds, so that $e^{-(s+t)\Gen}\subset e^{-t\Gen}e^{-s\Gen}$; but equality necessarily holds, as the
injection $e^{-t\Gen}e^{-s\Gen}$ cannot be a proper extension of the surjection $e^{-(s+t)\Gen}$.
Whence \eqref{sg-id}. 
For $t+s\ge0\ge s$ this yields $e^{-t\Gen}e^{-s \Gen}=e^{-(t+s)\Gen}e^{s \Gen}e^{-s \Gen}=e^{-(t+s)\Gen}$.
The case $-s>t\ge0$ is similar.

Also the commutation follows at once, for the semigroup property gives
\begin{equation}
  e^{s\Gen}e^{-t\Gen}= e^{-t\Gen}e^{t\Gen}e^{s\Gen}e^{-t\Gen} = e^{-t\Gen}e^{(s+t)\Gen}e^{-t\Gen} = e^{-t\Gen}e^{s\Gen}I_{R(e^{t\Gen})},
\end{equation}
where the right-hand side is a restriction of $e^{-t\Gen}e^{s \Gen}$.
Finally \eqref{incl-eq} yields \eqref{eq:domain_inclusion}.
\end{proof}

\begin{Remark}
 $D(e^{-t\Gen}e^{s\Gen})=D(e^{-(t-s)\Gen})$ holds in \eqref{comm-eq}, because
\eqref{sg-id} extends to negative $s$ as stated. \mbox{Hence \eqref{comm-eq}} is a strict inclusion if
the the first one in \eqref{eq:domain_inclusion} is so for all $t,t'$.
\end{Remark}

\subsubsection{Some Regularity Properties}
As a preparation we treat a few regularity questions for 
$s\mapsto e^{(t-s)\Gen}f(s)$, where the analytic operator function $E(s)=e^{(t-s)\Gen}$ has 
a singularity at $s=t$; cf.\ \eqref{AetA-est}. This will be controlled when
$f\in L_1(0,t;B)$.
 
That $Ef=e^{(t-\cdot)\Gen}f$ also is in $L_1(0,t;B)$ is undoubtedly known.
So let us recall briefly how to prove it strongly measurable, i.e., to find
a sequence of simple functions converging pointwise to $E(s)f(s)$ for a.e.\ $s \in[0,t]$; cf.~\cite{Yos80}.
Now $f$ can be so approximated by a sequence $(f_n)$, and $E$ can by its
{continuity} $[0,t[\,\to \B(B)$ also be approximated pointwise for $s<t$ by $E_n$
defined on each {subinterval} $[t(j-1)2^{-n},tj2^{-n}[\,$, $j=1,\dots,2^n$, as the value of $E$ at 
the left end point. Then $Ef=\lim_n E_nf_n$ on $[0,t]$ a.e.
Therefore $e^{(t-\cdot)\Gen}f \in L_1(0,t;B)$ follows directly from \eqref{Mo-ineq},
\begin{align}  \label{L1-est}
  \|e^{(t-\cdot)\Gen} f \|_{L_1(0,t;B)}\leq \int_0^t \|e^{(t-s)\Gen}\|\|f(s)\|\,ds \leq M e^{\omega t} \|f \|_{L_1(0,t;B)}.
\end{align}
Moreover, $\dualp{\eta}{e^{(t-\cdot)\Gen} f}$ is seen to be in $L_1(0,t)$ 
by majorizing with $\|e^{(t-s)\Gen}f(s)\|_{B} \|\eta\|_{B^*}$, for strong measurability implies weak
measurability; cf.\ Section IV.5 appendix in \cite{ReSi80}.

The main concern is to obtain a Leibniz rule for the derivative:
\begin{equation}  \label{eq:Leibniz_rule}
\partial_s( e^{(T-s)\Gen} w(s)) = (-\Gen)e^{(T-s)\Gen}w(s)+ e^{(T-s)\Gen}\partial_s w(s).
\end{equation}
For $w \in C^1(0,T; B)$ this is unproblematic for $s<T$:
$w(s+h)=w(s)+ h \partial_s w(s) +o(h)$, where $o(h)/h\to 0$ for $h \rightarrow 0$; and 
the operator is differentiable in $\B(B)$ for $s<T$, cf. Theorem \ref{Pazy-thm}, so that
$e^{(T-(s+h))\Gen} = e^{(T-s)\Gen} + h (-\Gen)e^{(T-s)\Gen} +o(h)$. Hence a multiplication of the two expansions
gives the right-hand side of \eqref{eq:Leibniz_rule} to the first order in $h$. 
The Leibniz rule is more generally valid in the vector distribution sense:

\begin{Proposition}  \label{prop:Leibniz_rule}
If $\Gen$ generates an analytic semigroup on a Banach space $B$ and $w \in H^1(0,T; B)$, then
the Leibniz rule \eqref{eq:Leibniz_rule} holds in $\D'(0,T;B)$.
\end{Proposition}

\begin{proof} 
It suffices to cover the case $\omega=0$, for the other cases then follow by applying the formula
to the semigroup $e^{-\omega t}e^{t\Gen}$ generated by $\Gen-\omega I$.
For $w \in H^1(0,T;B)$ the standard convolution procedure gives a sequence $(w_k)$ in 
$C^{1}([0,T]; B)$ such that 
\begin{equation}
  w_k \rightarrow w  \quad\text{in}\quad L_2(0,T;B),\qquad
  w'_k \rightarrow w' \quad\text{in}\quad L_{2, \text{loc}}(0,T; B). 
\end{equation}
For arbitrary $\phi \in C_0^{\infty}(\,]0,T[\,)$, we find using the Bochner inequality  that 
\begin{align}  \label{eq:semigroupLeibniz}
  \| \int_0^T e^{(T-s)\Gen} (w(s)-w_k(s)) \phi(s) \,ds \|_{B}  \leq C \|w(s)-w_k(s)\|_{L_2(0,T;B)}, 
\end{align}
with $C= M (\int_{\supp \phi} |\phi(s)|^2 \,ds)^{1/2}$, where $M$ is the constant
in \eqref{Mo-ineq}. 

Hence $e^{(T-s)\Gen}w_k \rightarrow e^{(T-s)\Gen} w$ in $\D'(0,T; B)$, so via the
$C^1$-case above, as $\partial_s$ is continuous in $\cal D'$, we get
\begin{equation}
  \begin{split}
 \partial_s(e^{(T-s)\Gen}w) &= \lim_{k \rightarrow \infty}( \partial_s(e^{(T-s)\Gen}w_k))   
\\
 &=\lim_{k \rightarrow \infty}( (-\Gen)e^{(T-s)\Gen}w_k) + \lim_{k \rightarrow
   \infty}(e^{(T-s)\Gen}\partial_s w_k)
 = (-\Gen)e^{(T-s)\Gen}w +  e^{(T-s)\Gen}\partial_s w.
  \end{split}
\end{equation}
Indeed, the last limits exist in $\D'(0,T;B)$ by the choice of $w_k$, for if
$\epsilon>0$ is small enough,
\begin{gather} 
  \| \int_{\supp \phi} e^{(T-s)\Gen} (w'(s)-w'_k(s)) \phi(s) \,ds \|_{B} \leq c
  \int_{\varepsilon}^{T-\varepsilon}\|w'(s)-w_k'(s)\|_{B}\,ds,
\\
  \| \int_0^T (-\Gen) e^{(T-s)\Gen} (w(s)-w_k(s)) \phi(s) \,ds \|_{B}  \leq  \tilde{C}
  \|w-w_k\|_{L_2(0,T; B)}
\end{gather}
with $\tilde{C} = (\int_{\supp \phi} \big|\frac{c \phi(s)}{T-s} \big|^2 \,ds)^{1/2}$,
using the bound on $(-\Gen)e^{(T-s)\Gen}$ in Theorem~\ref{Pazy-thm}.
\end{proof}

\section{Functional Analysis of Initial Value Problems}\label{ivp-sect}

Having set the scene  in Section~\ref{LaxM-ssect} by recalling elliptic Lax--Milgram operators $\A$ in Gelfand triples
$(V,H,V^*)$, 
we now discuss solutions of the classical initial value problem
\begin{equation}  \label{ivp-id}
\left\{
\begin{aligned}
  \partial_t u + \A u &= f   && \text{in } \D'(0,T;V^*) \\
  u(0)      &= u_0 && \text{in } H.
\end{aligned}
\right .
\end{equation}
By definition of vector distributions, the above equation means that for every
scalar test function $\varphi\in C_0^\infty(]0,T[)$ one has
$\dual{u}{-\varphi'}+\dual{\A u}{\varphi}=\dual{f}{\varphi}$ as an identity in $V^*$.

First we recall the fundamental theorem for vector functions from \cite[Lem.~III.1.1]{Tm}.
Further below, it will be crucial for obtaining a solution formula for \eqref{ivp-id}.

\begin{Lemma}\label{lem:Temam}
For a Banach space $B$ and $u, g \in L_1(a,b;B)$ the following are equivalent: 
\begin{itemize}
  \item[\rm{(i)}] $u$ is a.e. equal to a primitive function of $g$, i.e.\ for some vector $\xi\in B$
    \begin{equation}
       u(t) = \xi + \int_a^t g(s) \,ds \quad \text{for a.e.~ $t \in [a,b]$}.
     \end{equation}
   \item[\rm{(ii)}] For each test function $\phi \in C_0^{\infty}(]a,b[)$ one has
       $\int_a^b u(t) \phi'(t) \,dt = -\int_a^b g(t)\phi(t) \,dt$. 
  \item[\rm{(iii)}] For each $\eta$ in the dual space $B'$, 
            $\frac{d}{dt} \dualp{\eta}{u} = \dualp{\eta}{g}$ holds in $\D'(a,b)$.
\end{itemize}
In the affirmative case, $u'=g$ as vector distributions in $\D'(a,b;B)$ by \upn{(ii)}, the
right-hand side in \upn{(i)} is a continuous representative of $u$ such that $\xi = u(a)$ and
\begin{align}\label{sbe11-eq}
  \sup_{a \leq t \leq b} \|u(t)\|_B \leq (b-a)^{-1} \|u\|_{L_1(a,b;B)} + \|g\|_{L_1(a,b;B)}.
\end{align}
\end{Lemma}

\begin{Remark}  \label{rem:Temam}
Lemma \ref{lem:Temam} is proved in \cite{Tm}, except for the estimate \eqref{sbe11-eq}:
the continuous function $\|u(t)\|_B$ attains its minimum at some $t_0\in[a,b]$, so applying the
Bochner inequality in \mbox{(i)} and the Mean Value Theorem, 
\begin{align}
   \|u(t)\|_B \leq \|u(t_0)\|_B + |\int_{t_0}^t \|g(t)\|_B \,dt|
  \leq \frac1{b-a}\int_a^b \|u(t)\|_B \,dt+\int_a^b \|g(t)\|_B \,dt. 
\end{align}
This yields \eqref{sbe11-eq}, 
hence the Sobolev embedding $W^{1,1}(a,b;B)\hookrightarrow C([a,b];B)$.
If furthermore $u,g \in L_2(a,b;B)$, we get the Sobolev embedding
$H^1(a,b;B)\hookrightarrow C([a,b];B)$ similarly,
\begin{align}  \label{eq:addendumII}
  \sup_{a \leq t \leq b} \|u(t)\|_B \leq 
  (b-a)^{-1/2} \|u\|_{L_2(a,b;B)} + (b-a)^{1/2} \|g\|_{L_2(a,b;B)} \leq c \|u\|_{H^1(a,b;B)}.
\end{align}
\end{Remark}

Secondly we recall the Leibniz rule 
$\frac{d}{dt} \ip{f(t)}{g(t)} = \ip{f'(t)}{g(t)} + \ip{f(t)}{g'(t)}$ valid for $f, g \in C^1([0,T]; H)$.
The well-known generalization below was proved  in real vector spaces in
\cite[Lem.~III.1.2]{Tm} for $u =v$. We briefly extend this to the general
complex case, which we mainly use to obtain that $\partial_t |u|^2= 2\Re\dual{u'}{u}$, though also
$u\ne v$ will be needed.

\begin{Lemma}  \label{lem:Leibniz_dual}
If $u, v \in L_2(0,T ; V) \cap H^1(0,T; V^*)$, then $t \mapsto \ip{u(t)}{v(t)}$ 
is in $L_1(0,T)$ and 
\begin{equation}
  \frac{d}{dt} \ip{u}{v} = \dualp{u'}{v} + \overline{\dualp{v'}{u}} \quad\text{ in $\cal D'(0,T)$}.
\end{equation}
Furthermore, $u$ and $v$ have continuous representatives on $[0,T]$, i.e., $u,v\in C([0,T];H)$.
\end{Lemma}

\begin{proof}
Let $u, v \in L_2(0,T ; V)$ with distributional derivatives $u', v' \in L_2(0,T ; V^*)$.  
As in the proof of Proposition~\ref{prop:Leibniz_rule} we obtain $u_m\in C^\infty([0,T];V)$ such that 
$u_m\rightarrow u$ in $L_2(0,T;V)$ whilst $u_m' \rightarrow u'$ in
$L_{2,\text{loc}}(0,T;V^*)$. Similarily $v$ gives rise to $v_m$. 

By continuity of inner products, $t\mapsto\ip{u}{v}$ is measurable on $[0,T]$ for $u, v \in L_2(0,T;V)$, 
and $\int_0^T|\ip{u}{v} |\,dt < \infty$. Sesquilinearity yields $\ip{u_m}{v_m} \rightarrow \ip{u}{v}$ in
$L_1(0,T)$ for $m\to\infty$, while both $\dualp{u'_m}{v_m} \rightarrow \dualp{u'}{v}$ and
$\dualp{v'_m}{u_m} \rightarrow \dualp{v'}{u}$ hold in $L_{2,\text{loc}}(0,T)$, hence in $\D'(0,T)$.

As differentiation is continuous in $\D'(0,T)$, one finds from the $C^1$-case and \eqref{eq:HV*} that
\begin{align}
  \frac{d}{dt} \ip{u}{v} =\lim_m\frac{d}{dt} \ip{u_m}{v_m} = 
  \lim_m\ip{u'_m}{v_m} +\lim_m\overline{\ip{v'_m}{u_m}}
  = \dualp{u'}{v} + \overline{\dualp{v'}{u}}.
\end{align}
Taking $v=u$ the function $t\mapsto |u(t)|^2$ is seen to be in $W^{1,1}(0,T)\subset
C([0,T])$, and since any $u\in H^1(0,T;V^*)$ is continuous in $V^*$ by Remark~\ref{rem:Temam}, one can 
also here obtain from Lemma~III.1.4 in \cite{Tm} that $u\colon [0,T]\to H$ is continuous. Similarly for $v$.
\end{proof}

\subsection{Existence and Uniqueness}
  \label{sec:solvability_of_backward_heateq_existence_and_uniqueness_of_a_soln}
In our presentation the following result is a cornerstone, relying on the full framework in 
Section~\ref{LaxM-ssect}; in particular $A$ need not be selfadjoint:
 
\begin{Theorem}\label{thm:Temam}
Let $V$ be a separable Hilbert space with $V \subseteq H$ algebraically, topologically and densely,
cf. \eqref{VHV-eq} and
\eqref{eq:VV*emb},
and let $\A\colon V\to V^*$ be the bounded Lax--Milgram operator induced by a $V$-elliptic
sesquilinear form, cf.~\eqref{eq:calA}. When $u_0 \in H$ and $f \in L_2(0,T; V^*)$ are given, 
then \eqref{ivp-id} has a uniquely determined solution $u(t)$ belonging to the space 
\begin{align}  \label{eq:X}
  X = L_2(0,T; V) \bigcap C([0,T];H) \bigcap H^1(0,T;V^*).
\end{align}
\end{Theorem}

We omit a proof of this theorem, as it is a special case of a more general result of
Lions and Magenes~\cite[Sect.~3.4.4]{LiMa72} on $t$-dependent forms $a(t;u,v)$. Clearly the conjunction of 
$u\in L_2(0,T;V)$ and $u'\in L_2(0,T;V^*)$, which appears in \cite{LiMa72}, is equivalent to the claim
in \eqref{eq:X} that $u$ belongs to the intersection of $L_2(0,T,V)$ and $H^1(0,T;V^*)$.

Alternatively one can use Theorem~III.1.1 in Temam's book~\cite{Tm}, where proof is given using
Lemma~\ref{lem:Temam} to reduce to the scalar
differential equation $\partial_t \dualp{u}{\eta} + a(u,\eta)= \dualp{f}{\eta}$ in $\D'(0,T)$, for
$\eta \in V$, which is treated by Faedo--Galerkin approximation and basic functional
analysis. His proof extends straightforwardly, from a specific triple $(H,V,a)$ for the Navier-Stokes
equations, to the general set-up in Section~\ref{LaxM-ssect}, also when $A^*\ne A$.

However, either way, we need the finer theory described in the next two subsections.

\subsection{Well-Posedness}
We now substantiate that the unique solution from Theorem~\ref{thm:Temam} depends continuously on the data, 
so that \eqref{ivp-id} is well-posed in the sense of Hadamard. 
First we note that the solution in Theorem~\ref{thm:Temam}
is an element of the space $X$ in \eqref{eq:X},
which is a Banach space when normed, as done throughout, by
\begin{align}  \label{eq:Xnorm}
\|u\|_{X} = \big(\|u\|^2_{L_2(0,T;V)} + \sup_{0 \leq t \leq T}|u(t)|^2 + \|u\|^2_{H^1(0,T;V^*)}\big)^{1/2}.
\end{align}
To clarify a redundancy in this choice, we need a Sobolev inequality for vector functions. 

\begin{Lemma} \label{lem:Sobolev}
  There is an inclusion $L_2(0,T;V)\cap H^1(0,T;V^*)\subset C([0,T];H)$ and
  \begin{equation}
  \sup_{0\le t\le T}| u(t)|^2\le (1+\frac{C_2^2}{C_1^2T})\int_0^T \|u\|^2\,dt+\int_0^T \|u'\|_*^2\,dt.
  \end{equation}
\end{Lemma}
\begin{proof}
  If $u$ belongs to the intersection, the continuity follows from
  Lemma~\ref{lem:Leibniz_dual},  where the formula gives $\partial_t|u|^2= 2\Re
  \dual{u'}{u}$. By Lemma~\ref{lem:Temam}, integration of both sides entails
\begin{equation}
  | u(t)|^2\le |u(t_0)|^2+\int_0^T (\|u\|^2+ \|u'\|_*^2)\,dt,
\end{equation}
which by use of the Mean Value Theorem as in Remark~\ref{rem:Temam} leads to the estimate.
\end{proof}
\begin{Remark} \label{redundancy-rem}
  In our solution set $X$ in \eqref{eq:X} one can safely omit the space $C([0,T];H)$,
  according to Lemma~\ref{lem:Sobolev}. Likewise $\sup|u|$ can be removed from $\| \cdot\|_{X}$,
  as one just obtains an equivalent norm (similarly for the term $\int_0^T\|u(t)\|_*^2\,dt$ in \eqref{eq:X-intro}). 
  Thus $X$ is more precisely a Hilbertable space; we omit this detail in the sequel for the sake
  of simplicity. However,
  we shall keep $X$ as stated in order to emphasize the properties of the solutions.
\end{Remark}

The next result on stability is well known among experts, and while it may be
derived from the abstract proofs in 
\cite{LiMa72}, we shall give a direct proof based on explicit estimates:

\begin{Corollary}\label{cor:Temam}
The unique solution $u$ of \eqref{ivp-id}, given by
Theorem \ref{thm:Temam}, depends
continuously as an element of $X$ on the data $(f,u_0) \in L_2(0,T;V^*)\oplus H$, i.e.
\begin{align}\label{ivp-est}
\|u\|^2_{X} \leq c (|u_0|^2 + \|f\|^2_{L_2(0,T;V^*)}).
\end{align}
That is, the solution operator $(f,u_0)\mapsto u$ is a bounded linear map
$L_2(0,T;V^*)\oplus H\to X$.
\end{Corollary}

\begin{proof}
Clearly $u \in L_2(0,T;V)$ while the functions $u',f$ and $\A u$ belong to $ L_2(0,T;V^*)$, so as an identity of
\mbox{integrable functions,}
\begin{align}
  \Re\dualp{\partial_t u}{u} + \Re\dualp{\A u}{u} = \Re\dualp{f}{u}.
\end{align}
Hence Lemma \ref{lem:Leibniz_dual} and the $V$-ellipticity gives
\begin{align}
  \partial_t |u|^2 +  2C_4 \|u\|^2 \leq 2 |\dualp{f}{u}| \leq C_4^{-1} \|f\|_{*}^2 + C_4 \|u\|^2.
\end{align}

Using again that $|u(t)|^2$ and $\partial_t |u(t)|^2 $ are in $L_1(0,T)$, taking $B=\C$ in Lemma \ref{lem:Temam} yields
\begin{align}
  |u(t)|^2 + C_4 \int_0^t \|u(s)\|^2 \,ds \leq |u_0|^2 + C_4^ {-1}\|f\|_{L_2(0,T;V^*)}^2.
\end{align}
For the first two contributions to the $X$-norm this gives
\begin{align}
  \sup_{0\leq t \leq T} |u(t)|^2  &\leq |u_0|^2 + C_4^{-1}\|f\|_{L_2(0,T;V^*)}^2,
\\
  \|u\|_{L_2(0,T;V)}^2 &\leq C_4^{-1}|u_0|^2 + C_4^{-2}\|f\|_{L_2(0,T;V^*)}^2.
\label{uV-est}
\end{align}
Since $u$ solves \eqref{ivp-id} it is clear that
$\|\partial_t u(t) \|_{*}^2 \leq (\|f(t)\|_{*} + \|\A u \|_{*})^2 $,
so we get
\begin{align}
  \int_0^T \|\partial_t u(t) \|_{*}^2 \,dt \leq 2 \int_0^T \|f(t)\|_{*}^2 \,dt 
  + 2 \|\A\|_{\B(V,V^*)}^2  \int_0^T \|u\|^2 \,dt ,
\end{align}
which upon substitution of \eqref{uV-est} altogether shows \eqref{ivp-est}.
\end{proof}

\subsection{The First Order Solution Formula}  \label{sec:solvability_of_backward_heateq_a_repr_of_the_soln}
We now supplement the  well-posedness by a direct proof of the variation of constants formula,
which requires that the extended Lax--Milgram operator $\A$ generates an analytic semigroup in $V^*$.
This is known, cf.~\cite{Tan79}, but lacking a concise proof in the literature,  we
begin by analysing $A$ in $H$:

\begin{Lemma}  \label{Agen-lem}
For a $V$-elliptic Lax--Milgram operator $A$,
both $-A$ and $-A^*$ have the sector $\Sigma$ in \eqref{eq:Sigma} in their resolvent sets for
$\theta=\operatorname{arccot}(C_3/C_4)$ and they generate analytic semigroups on $H$. 
This holds verbatim for the extensions $-\A$ and $-\A'$ in $V^*$.
\end{Lemma}

\begin{proof}
To apply Theorem~\ref{Pazy-thm}, we let $\lambda\neq0$ be given in
the sector $\Sigma$ for some angle $\theta$ satisfying $0< \theta< \operatorname{arccot}(C_3/C_4)$.
Then it is clear that $\delta = -\operatorname{sgn}(\Im\lambda)\theta$ or $\delta = 0$ gives
\begin{equation}  \label{dl-ineq}
  \Re (e^{i \delta} \lambda) \geq 0.
\end{equation}
In case $\delta \in \left\{ \pm \theta \right\}$ a multiplication of the inequalities 
\eqref{Asect-id} by $-\sin \delta$ yields
\begin{align}
- \sin \delta \Im a(u,u) \geq - C_3 C_4^{-1} \sin \theta \Re a(u,u).
\end{align}
In addition $C_{\theta} := C_4 \cos \theta - C_3 \sin \theta > 0$, because $\cot \theta  > C_3C_4^{-1}$.
So for $u \in D(A)$,
\begin{align}
   \Re (e^{i\delta}(a(u,u)+\lambda\ip{u}{u})) &\geq \Re (e^{i \delta} a(u,u)) =
   \cos \delta \Re a(u,u) - \sin \delta \Im a(u,u)
\notag \\
  & \geq (\cos \theta - C_3 C_4^{-1} \sin \theta) \Re a(u,u)
\notag \\
  & \geq   C_{\theta} \|u\|^2.
\label{eq:roteret_V-ell}
\end{align}
This $V$-ellipticity holds also if $\delta=0$, cf.\ \eqref{dl-ineq},
so $e^{i \delta}(A+\lambda I)$ is in any case bijective; and so is $-A-\lambda I$.

To bound $-(A+\lambda I)^{-1}$, we see from \eqref{eq:roteret_V-ell} 
that for $u \in D(A)$, 
\begin{align}
  |\lambda|\ip{u}{u} & \leq |\ip{(A+\lambda)u}{u} | + |a(u,u)| \leq |\ip{(A+\lambda)u}{u} | + C_3\|u\|^2 
\notag \\
    & \leq (1+ C_3 C_{\theta}^{-1}) |\ip{(A+\lambda)u}{u}|.
\label{eq:solvability_of_backward_heateq_vurdering_med_lambda_og_h_norm}
\end{align} 
This implies \eqref{lIA-est} for $-A$.
Since $A^*$ is the Lax--Milgram operator associated to the elliptic form $a^*$, the above also
entails the statement for $-A^*$.

For $\A$ it follows at once from \eqref{eq:roteret_V-ell} that 
$\Re \dualp{e^{i \delta} (\A + \lambda) u}{u} \geq C_{\theta} \|u\|^2$  for $u \in V$.  
Hence $R(\A+\lambda I)$ is closed in $V^*$, and it is also dense since $R(\A+\lambda I)\supset R(A+\lambda I)=H$ 
by the above; i.e., $\A+\lambda I$ is surjective.
Mimicking \eqref{eq:solvability_of_backward_heateq_vurdering_med_lambda_og_h_norm}, 
we get for $u\neq 0$, $\|w\|=1$, both in $V$,
\begin{align} \label{AV-est}
  |\lambda|\cdot \|u\|_{*}  
  \leq \sup_w\big|\dualp{(\A +\lambda)u}{w}\big| + 
   C_3 C_{\theta}^{-1} \big| \dualp{(\A +\lambda)u}{\frac1{\|u\|}u}\big|
  \leq  c\|(\A +\lambda)u\|_{*}.
\end{align}
This yields injectivity of $\A+\lambda I$ and the resolvent estimate.
$\A'$ is covered through $a^*$.
\end{proof}

We denote by $e^{-t\A}$ the semigroup generated by $-\A$ on $V^*$, to distinguish it
from $e^{-tA}$ on $H$. Analogously for $e^{-t\A'}\in\B(V^*)$.
As $A\subset\A$  implies that
$(\A+\lambda I)^{-1}|_{H}=(A+\lambda I)^{-1}$, and since $A$ and $\A$ have the same sector $\Sigma$ by 
Lemma~\ref{Agen-lem}, the well-known  Laplace transformation formula, \mbox{cf.\ \cite[Thm.~1.7.7]{Paz83},} yields the
corresponding fact, say $e^{-t\A}|_{H}=e^{-tA}$ for the semigroups:

\begin{Lemma} \label{lem:AA}
For all $x\in H$ one has $e^{-t\A}x=e^{-tA}x$ as well as $e^{-t\A'}x=e^{-tA^*}x$.
\end{Lemma}

We could add that $A$ and $A^*$ are dissipative, as $m(A)>0$, $m(A^*)>0$ in $H$, so $e^{-tA}$, $e^{-tA^*}$ are
contractions for $t\ge 0$ by the Lumer--Philips theorem; cf.\ \cite[Cor.~14.12]{G09}.

Using Lemmas~\ref{Agen-lem} and \ref{lem:AA}, the announced formula results as an addendum to
Theorem~\ref{thm:Temam}:

\begin{Theorem}   \label{thm:repr_for_u}
The unique solution $u$ in $X$ provided by Theorem~\ref{thm:Temam} satisfies that
\begin{equation} \label{u-id}
  u(t) = e^{-tA}u_0 + \int_0^t e^{-(t-s)\A}f(s) \,ds \qquad\text{for } 0\leq t\leq T,
\end{equation}
where  each of the three terms belongs to $X$.
\end{Theorem}
\begin{proof}
Once \eqref{u-id} has been shown, Theorem~\ref{thm:Temam} applies in particular to cases with $f=0$, yielding that
$u(t)$ and hence $e^{-tA}u_0$ belongs to $X$. For general data $(f,u_0)$ this means that the last
term containing $f$ necessarily is a member of $X$ too.

To derive formula \eqref{u-id} in the present general context, one should note that all terms in the equation
$\partial_t u+\A u=f$ belong to the space $L_2(0,T;V^*)$. Therefore the operator $e^{-(T-t)\A}$
applies to both sides as an integration factor, yielding
\begin{equation}
  e^{-(T-t)\A}\partial_t u(t)+ e^{-(T-t)\A}\A u(t)=e^{-(T-t)\A}f(t).
\end{equation}
Now $e^{-(T-t)\A}u(t)$ belongs to $L_1(0,T;V^*)$, cf.\ the argument prior to \eqref{L1-est}.
For its derivative in $\D'(0,T;V^*)$ the Leibniz rule in
Proposition~\ref{prop:Leibniz_rule} gives, as $u(t)\in V=D(\A)$ for $t$ a.e.,
\begin{equation}
  \partial_t(e^{-(T-t)\A}u(t))=e^{-(T-t)\A}\partial_tu(t)+ e^{-(T-t)\A} \A u(t).
\end{equation}
As both terms on the right-hand side are in $L_2(0,T;V^*)$, the implication \mbox{(ii)}$\implies$\mbox{(i)}
in Lemma~\ref{lem:Temam} gives that
\begin{equation} \label{eq:identityT}
  e^{-(T-t)\A}u(t)=e^{-T\A}u_0+\int_0^t e^{-(T-s)\A}f(s)\,ds.
\end{equation}
From this identity in $C([0,T];V^*)$ formula \eqref{u-id} results in case
$t=T$ by evaluation, when also Lemma~\ref{lem:AA} is used for the term containing $u_0$.
However, obviously the above argument applies to any subinterval 
$[0,T_1]\subset [0,T]$, whence \eqref{u-id} is valid for all $t$ in $[0,T]$.
\end{proof}

Alternatively one could conclude by applying $e^{-(T-s)\A}=e^{-(T-t)\A}e^{-(t-s)\A}$ 
in \eqref{eq:identityT} and use the Bochner identity to commute $e^{-(T-t)\A}$ with the
integral: as analytic semigroups like $e^{-(T-t)\A}$ are always injective, cf.~Proposition~\ref{inj-prop},
formula \eqref{u-id} then results at once. 

For later reference we show similarly the next inequality:

\begin{Corollary} \label{cor:u-inequality}
  The solution $e^{-t\A}u_0$ to the problem with $f=0$ in Theorem~\ref{thm:Temam} belongs to $L_2(0,T;V)$ and
  fulfils, for every $u_0\in H$,
  \begin{equation}
    \sup_{0\le t\le T}(T-t)|e^{-t\A}u_0|^2 \le C_5\int_0^T \| e^{-t\A}u_0\|^2 \,dt.
  \end{equation}
\end{Corollary}
\begin{proof}
  It is seen from Theorem~\ref{thm:repr_for_u} that $u(t)=e^{-t\A}u_0$ always is in $L_2(0,T;V)$, as a
  member of $X$. By taking scalar products with $(T-\cdot)u$ on both sides of the differential
  equation, one obtains in $L_1(0,T)$ the identity
  \begin{equation}
    (T-t)\dual{u'(t)}{u(t)}+(T-t)a(u(t),u(t))=0.
  \end{equation}
Taking real parts here, applying Lemma~\ref{lem:Leibniz_dual} to $u$ and integrating partially on $[t,T]$, one obtains
\begin{equation}
  \int_t^T |u(s)|^2\,ds -(T-t)|u(t)|^2= -2\int_t^T (T-s)\Re a(u(s),u(s))\,ds.
\end{equation}
By reorganising this, a crude estimate yields the result at once for $C_5=\frac{C_2}{C_1}+2TC_3$.
\end{proof}

\subsection{Non-Selfadjoint Dynamics}
It is classical that $e^{-tA}u_0$ in \eqref{u-id} is a term that decays exponentially for
$t\to\infty$ if $A$ is self-adjoint and has compact inverse on $H$. 
This follows from the eigenfunction expansions, cf.\ the formulas in the introduction and
Section~\ref{sa-ssect}, which imply for the 'height' function
$h(t)=|e^{-tA}u_0|$ that \mbox{$h(t)={\cal O}(e^{-t\Re\lambda_1})$}.

However, it is a much more precise dynamical property that $h(t)$ is a \emph{strictly convex}
function for $u_0\ne0$ (we refer to \cite{NiPe06} for a lucid account of convex
functions). Strict convexity is established below for wide classes of non-self-adjoint $A$, namely
if $A$ is hyponormal or such that $A^2$ is accretive. 

Moreover, it seems to be a novelty that the \emph{injectivity} of $e^{-tA}$ provided by Proposition~\ref{inj-prop} 
implies the strict convexity. For simplicity we first explain this for the square $h(t)^2$.

Indeed, differentiating twice for $t>0$ one finds for $u=e^{-tA}u_0$,
\begin{equation} \label{h''-id}
  (h^2)''= (-2\Re\ip{Ae^{-tA}u_0}{e^{-tA}u_0})'=2\Re\ip{A^2u}{u}+2\ip{Au}{Au}. 
\end{equation}
In case $A^2$ is accretive, that is when $m(A^2)\ge0$, we may keep only
the last term in \eqref{h''-id} to get that $(h^2)''(t)\ge 2|Ae^{-tA}u_0|^2$, which for $u_0\ne0$ implies $(h^2)''>0$ as 
both $A$ and $e^{-tA}$ are injective; cf.\ \eqref{eq:Sigma} and Proposition~\ref{inj-prop}.
Hence $h^2$ is strictly convex for $t>0$ if $m(A^2)\ge0$.

Another case is when $A$ is \emph{hyponormal}. For an unbounded operator $A$ this means that
\begin{equation}  \label{eq:hyponormal}
  D(A)\subset D(A^*)\quad\text{with} \quad  |A^*u|\le |Au| \quad\text{for all $u\in D(A)$}.
\end{equation}
Cf.\ the work of Janas~\cite{Jan94}.
Note that if both $A$, $A^*$ are hyponormal, then $A$ is normal.

This is a quite general class, but it fits most naturally into the present discussion:
For hyponormal $A$ we have $R(e^{-tA})\subset D(A)\subset D(A^*)$, which shows that $A^*e^{-tA}u_0$
is defined. Using this and hyponormality once more in \eqref{h''-id}, we get
\begin{equation}
  (h^2)''(t)\ge\ip{Au}{A^*u}+\ip{A^*u}{Au}+|Au|^2+|A^*u|^2 =|(A+A^*)e^{-tA}u_0|^2.
\end{equation} 
Now $(h^2)''>0$ follows for $u_0\ne0$ from injectivity of $e^{-tA}$ and of $A+A^*$; the latter
holds since $2a_{\Re}$ is $V$-elliptic. So $h^2$ is also strictly convex for hyponormal $A$. 

Also on the closed half-line with $t\ge0$ there is a result on non-selfadjoint dynamics.
Here we return to $h(t)$ itself and normalise, at no cost, to $|u_0|=1$ to get cleaner statements:

\begin{Proposition}  \label{sc-prop}
  Let $A$ denote a $V$-elliptic Lax--Milgram operator, defined from a triple $(H,V,a)$, such that $A$ is
  hyponormal, as above, or such that $A^2$ is accretive,  
  and let $u$ be the solution from Theorem~\ref{thm:Temam} for $f=0$ and $|u_0|=1$.
  Then $h(t)=|u(t)|$ is strictly decreasing and \emph{strictly convex} for $t\ge0$ and 
  differentiable from the right at $t=0$ with 
  \begin{equation}
    h'(0)=-\Re \ip{Au_0}{u_0} \qquad \text{for $u_0\in D(A)$},  
  \end{equation}
  and generally
\begin{equation}
  h'(0)\le -m(A).
\end{equation}
\end{Proposition}
\begin{Remark}
  The derivative $h'(0)$ might be $-\infty$ if $u_0\in H\setminus D(A)$.
\end{Remark}

\begin{proof}
By the convexity shown above, $(h^2)'$ is increasing.
Since $m(A)>0$ holds by the $V$-ellipticity, $h^2$ is strictly decreasing (and so is $h$) for $t>0$ as
  \begin{equation}
   (h^2)'(t)=-2\Re\ip{Ae^{-tA}u_0}{e^{-tA}u_0}\le-2m(A)|e^{-tA}u_0|^2<0.
  \end{equation}
These properties give that $h'=(h^2)'/(2\sqrt{h^2})$ is \emph{strictly} increasing for $t>0$, so the Mean Value Theorem
yields that $(h(t)-h(s))/(t-s)<(h(u)-h(t))/(u-t)$ for $0<s<t<u$; that is, $h$ is strictly convex
on $\,]0,\infty[\,$. 

The inequality $h((1-\theta)t+\theta s)\le (1-\theta)h(t)+\theta h(s)$,
$\theta\in\,]0,1[\,$ now extends by continuity to $t=0$. 
So does strict convexity of $h$, using twice that the slope function is increasing. 

 By convexity $h'$ is increasing for $t>0$, so $\lim_{t\to0^+} h'(t)=\inf h' \ge-\infty$.
  For each $0<s<1$ the continuity of $h$ yields $|e^{-tA}u_0|\ge s|u_0|=s$ for all sufficiently small $t\ge0$.
  By the above {formulas for $h'$ and $(h^2)'$ we have} $h'(t)=-\Re\ip{Ae^{-tA}u_0}{e^{-tA}u_0}/|e^{-tA}u_0|$,
  so the Mean Value Theorem gives for some $\tau\in\,]0,t[\,$,
  \begin{equation}
    t^{-1}(h(t)-h(0))=h'(\tau)\le -m(A) s<0. 
  \end{equation}
  Hence $h(0)> h(t)$ for all $t>0$. Moreover, the limit of $h'(\tau)$ was shown
  above to exist for $\tau\to0^+$, so $h'(0)$ exists in $[-\infty,-m(A)]$.  
  If $u_0\in D(A)$ we may commute $A$ with the semigroup in the formula for $h'(\tau)$, which by continuity gives
  $h'(0)=-\Re\ip{Au_0}{u_0}$.
  \end{proof}

Proposition~\ref{sc-prop} is a stiffness result for $u=e^{-tA}u_0$, due to strict convexity of 
$|e^{-tA}u_0|$. It is noteworthy that when $A\ne A^*$, then Proposition~\ref{sc-prop} gives conditions under which
the eigenvalues in $\C\setminus\R$ (if any) never lead to oscillations in the \emph{size} of the solution.

\begin{Remark}
Since $h'(0)$ is estimated in terms of the lower bound $m(A)$, it
is  the numerical range $\nu(A)$, rather 
than $\sigma(A)$, that controls \emph{short-time} decay of the solutions $e^{-tA}u_0$.
\end{Remark}

\begin{Remark}\label{A2acc-rem} In Proposition~\ref{sc-prop} we note that when $A^2$ is accretive, i.e., $m(A^2)\ge0$,
  then $A$ is  \emph{necessarily}  
sectorial with half-angle $\pi/4$; that is $\nu(A)\subset\big\{z\in\C\bigm| |\arg(z)|\le\pi/4\big\}$. This
may be seen as in \cite[Lem.~3]{Sho74}, where reduction to bounded operators was
made in order to invoke the operator monotonicity of the square root. 
\end{Remark}

\begin{Remark} \label{Show-rem}
We take the opportunity to point out an error in  (\cite[]{Sho74}, Lemma~3), where it incorrectly
was claimed that having half-angle $\pi/4$ also is sufficient for $m(A^2)\ge0$.   
A counter-example is available  already for $A$ in $\B(H)$ (if $\operatorname{dim} H\ge2$), as 
$A=X+\im Y$ for self-adjoint $X$, $Y\in\B(H)$:  
here $m(A)\ge0$ if and only if $X\ge0$, and we can arrange that $A$ has half-angle $\pi/4$,
that is $|\Im\ip{Av}{v}|\le\Re\ip{Av}{v}$ or $|\ip{Yv}{v}|\le\ip{Xv}{v}$, 
by designing $Y$ so that $-X\le Y\le X$.
Here we may take 
$Y=\delta X+\lambda_1 U$, where $\delta>0$ is small enough and $U$ is a partial isometry that interchanges two 
eigenvectors $v_1$, $v_2$ of $X$ with eigenvalues $\lambda_2>\lambda_1>0$, $U=0$ on
$H\ominus\Span(v_1,v_2)$. In fact, writing $v=c_1v_1+c_2v_2+v_\perp$ for $v_\perp\in
H\ominus\Span(v_1,v_2)$, since $v_1\perp v_2$, the above inequalities for $Y$ are equivalent to 
$2\lambda_1|\Re(c_1\bar c_2)|\le \lambda_1(1-\delta)|c_1|^2+(1-\delta)\lambda_2|c_2|^2
+(1-\delta)\ip{Xv_\perp}{v_\perp}$, which by the positivity of $X$ and Young's inequality is
implied by $1/(1-\delta)\le (1-\delta)\frac{\lambda_2}{\lambda_1}$, that is if $0<\delta\le
1-\sqrt{\lambda_1/\lambda_2}$. 
Now, $m(A^2)\ge0$ if and only if $|Xv|^2\ge |Yv|^2$ for all $v$ in $H$, but 
this will \emph{always} be violated, as one can see from
$|Yv|^2= \delta^2|Xv|^2+\lambda_1^2|Uv|^2+2\delta\lambda_1\Re\ip{Xv}{Uv}$ by inserting $v=v_1$, 
for the last term drops out as $v_1\perp v_2=Uv_1$, so that actually $|Yv_1|^2= (\delta^2+1)\lambda_1^2> |Xv_1|^2$.
Thus
$A= \left(\begin{smallmatrix}\lambda&0\\0&4\lambda\end{smallmatrix}\right) + i\lambda\left(\begin{smallmatrix}\delta& 1\\ 1 & 4\delta\end{smallmatrix}\right)$
is a counter-example in $\C^2$ for any $\lambda>0$, $0<\delta\le 1/2$.
\end{Remark}

\begin{Remark}
It is perhaps useful to emphasize the benefit from joining the two methods. Within semigroup theory
the ``mild solution'' given in \eqref{u-id}
is the only possible solution to \eqref{ivp-id}; but as
our class of solutions is larger, the extension of the old uniqueness argument in
Theorem~\ref{thm:repr_for_u} was needed. 
Existence of a solution is for analytic semigroups classical if 
$f\colon [0,T]\to H$ is H{\"o}lder continuous, cf.\ \cite[Cor.~4.3.3]{Paz83}.
Using functional analysis, this gap to the weaker condition $f \in L_2(0,T;V^*)$
is bridged by Theorem~\ref{thm:repr_for_u}, which states that the mild solution is indeed the
solution in the space of vector distributions in Theorem~\ref{thm:Temam}; albeit
at the expense that the generator $A$ is a $V$-elliptic Lax--Milgram operator. 
\end{Remark}

\section{Abstract Final Value Problems}\label{fvp-sect}
In this section, we show for a Lax--Milgram operator $\A$ that the final value problem 
\begin{equation}\label{eq:fvp}
\left\{
\begin{aligned}
  \partial_t u+ \A u &= f &&\quad\text{ in $\cal D'(0,T;V^*)$}, 
\\
  u(T) &=u_T &&\quad\text{ in $H$},
\end{aligned}
\right .
\end{equation}
is \emph{well-posed} when the final data belong to an appropriate space, to be identified
below. This is obtained via comparison with the initial value problem treated in Section~\ref{ivp-sect}.

\subsection{A Bijection From Initial to Terminal States}
\label{bijection-ssect}
According to Theorem~\ref{thm:Temam}, the solutions to the differential equation $u'+\A u=f$  are
for fixed $f$ parametrised by the initial states $u(0)\in H$. To study the terminal states
$u(T)$  we note that \eqref{u-id} yields 
\begin{align}\label{eq:u_T}
  u(T) = e^{-TA}u(0) + \int_0^T e^{-(T-s)\A}f(s) \,ds.
\end{align}

This representation of $u(T)$ is essential in what follows, as it gives a bijective correspondence
$u(0)\leftrightarrow u(T)$ between the initial and terminal states, as accounted for below.

First we analyse the integral term above by introducing the yield map $f\mapsto y_f$ given by
\begin{align}  \label{yf-id}
  y_f =  \int_0^T e^{-(T-s)\A}f(s) \,ds, \qquad f\in L_2(0,T;V^*). 
\end{align}
Clearly $y_f$ is a vector in $V^*$ by definition of the integral (and since $C([0,T];V^*)\subset
L_1(0,T;V^*)$). But actually it is in the smaller space $H$, 
for $y_f =u(T)$ holds in $H$ when $u$ is the solution for $u_0 =0$ of \eqref{ivp-id}, and then 
Corollary~\ref{cor:Temam} yields an estimate of $\sup_{t \in [0,T]} |u(t)|$ by the $L_2$-norm of $f$;
cf.~\eqref{eq:Xnorm}. \mbox{In particular, we have}
\begin{align}
  |y_f|  \leq c \|f\|_{L_2(0,T;V^*)}.
\label{eq:yfH}
\end{align}

Moreover, $f \mapsto y_f$ is by \eqref{eq:yfH} bounded $L_2(0,T;V^*)\to H$, and it has dense range
in $H$ containing all $x\in D(e^{\varepsilon A})$ for
every $\varepsilon>0$, for if in \eqref{yf-id} we insert the piecewise continuous function 
\begin{equation} \label{fe-id}
  f_\varepsilon(s)=\1_{[T-\varepsilon,T]}(s)e^{(T-\varepsilon-s)A}(\frac1\varepsilon e^{\varepsilon A}x),   
\end{equation}
then the semigroup property gives
$y_{f_\varepsilon}= \int_{T-\varepsilon}^T e^{-\varepsilon A}(\frac1\varepsilon e^{\varepsilon A}x)\,ds = 
       \frac1\varepsilon\int_{T-\varepsilon}^T  x\,ds= x$.
However, standard operator theory gives the optimal result, that is, surjectivity:

\begin{Proposition} \label{yf-prop}
  The yield map $f\mapsto y_f$ is in $\B(L_2(0,T;V^*),H)$ and it is surjective, $R(y_f)=H$.
  Its adjoint in $\B(H,L_2(0,T;V))$ is the orbit map given by $v\mapsto e^{-(T-\cdot)A^*}v$.
\end{Proposition}
\begin{proof}
To determine the adjoint of $f\mapsto y_f$, we first calculate for $f\in L_2(0,T;H)$ so that
the integrand in \eqref{yf-id} belongs to $C([0,T];H)$. For $v\in H$ we get, using the  Bochner identity twice,
\begin{equation}
  \ip{y_f}{v} = \int_0^T \ip{e^{-(T-s)A}f(s)}{v} \,ds = \int_0^T \ip{f(s)}{e^{-(T-s)A^*}v} \,ds
  =\dual{f}{e^{-(T-s)A^*}v}. 
\end{equation}
The last scalar product makes sense because $s\mapsto e^{-(T-s)A^*}v$ is in $L_2(0,T;V)$, as
seen by  applying Corollary~\ref{cor:u-inequality} to the Lax--Milgram operator $A^*$, and
$L_2(0,T;V)$ is the dual space to $L_2(0,T;V^*)$; {cf.~Remark}~\ref{dual-rem} below. 
Since $L_2(0,T;H)$ is dense in $L_2(0,T;V^*)$, it follows by closure that the left- and
right-hand sides are equal for every $f\in L_2(0,T;V^*)$ and $v\in H$. Hence $v\mapsto
e^{-(T-\cdot)A^*}v$ is the adjoint of $y_f$. 

Applying Corollary~\ref{cor:u-inequality} to $A^*$ for $t=0$, a change of variables yields for every $v\in H$,
\begin{equation}
      |v|^2 \le \frac{C_5}{T}\int_0^T \| e^{-(T-s)A^*}v\|^2 \,ds.
\end{equation}
This estimate from below of the adjoint is equivalent to closedness of the range of $y_f$, as the range is  
dense by \eqref{fe-id}. This follows from the Closed Range Theorem; cf.\ \cite[Thm.~3.1]{johnsen:00}
for a general result \mbox{on this.} 
\end{proof}

\begin{Remark} \label{dual-rem}
{The Banach spaces} $L_2(0,T;V)$, $L_2(0,T;V^*)$ are in duality, and $L_2(0,T;V)^*$ identifies with $L_2(0,T;V^*)$: 
for each $\Lambda\in L_2(0,T;V)^*$ the inner product $a_{\Re}$ 
and Riesz' theorem yield $h\in L_2(0,T;V)$ that for $g\in L_2(0,T;V)$ fulfils
$\dual{\Lambda}{g}=\int_0^Ta_{\Re}(h,g)\,dt$; so $\dual{\Lambda}{g}=\int_0^T\dual{f}{g}\,dt$ for 
$f=\frac12(\A+\A')h$ in $L_2(0,T;V^*)$; cf.\ \eqref{eq:3norm} and \eqref{eq:calA}.
\end{Remark}

The surjectivity of $y_f$ can be shown in important cases using an explicit construction, which is
of interest in control theory (cf.\ Remark~\ref{yf-rem}), and given here for completeness: 

\begin{Proposition}  \label{Ryf-prop}
  If $A^*=A$ and $A^{-1}$ is compact, every $v\in H$ equals $y_f$ for some computable  $f\in  L_2(0,T;V^*)$.
\end{Proposition}
\begin{proof} 
Fact~\ref{fact:1} yields an ortonormal basis $(e_n)_{n\in\N}$ so that $Ae_n=\lambda_ne_n$,
hence any $v$ in $H$ fulfils $v = \sum_{j}\alpha_j e_j$ with $\sum_{j}|\alpha_j|^2<\infty$. 
By Fact~\ref{fact:2} every $f \in L_2(0,T;V^*)$ has an expansion 
\begin{equation}
  \label{eq:f_expansion}
  f(t) = \sum_{j=1}^{\infty} \beta_j(t) e_j = \sum_{j=1}^{\infty} \dualp{f(t)}{e_j} e_j
\end{equation}
converging in $V^*$ for $t$ a.e. Since $e^{-(T-s) \A}e_j = e^{-(T-s)\lambda_j} e_j$,
cf.~Remark~\ref{rem:bijektiv_onb}, such $f$ fulfill
\begin{align}
  y_f=\int_0^T e^{-(T-s) \A} f(s) \,ds 
  = \sum_{j=1}^{\infty}  e^{-T\lambda_j}(\int_0^T \beta_j(s)e^{s\lambda_j} \,ds) e_j.  
\end{align}
Hence $y_f = v$ is equivalent to the validity of
$\int_0^T\beta_j(s)e^{s\lambda_j}\,ds = \alpha_j e^{T\lambda_j}$ for $j\in\N$.
So if, in terms of some $\theta_j\in\,]0,1[\,$ to be determined, we take the coefficients of $f(t)$ as
\begin{equation} \label{betaj-id}
  \beta_j(t)=k_j\1_{[\theta_j T,T]}(t)\exp(t(\sqrt{\lambda_j}-\lambda_j)),
\end{equation}
then the condition will be satisfied if and only if 
$k_j=\alpha_je^{T\lambda_j}\sqrt{\lambda_j}(e^{T\sqrt{\lambda_j}}-e^{\theta_j T\sqrt{\lambda_j}})^{-1}$.

Moreover, using the equivalent norm $\vvvert\cdot\vvvert_{*}$ on $V^*$ in Fact~\ref{fact:2}, 
\begin{align}
  \|f\|^2_{L_2(0,T;V^*)}  = \int_0^T \vvvert f(t)\vvvert_{*}^2 \,dt =  \sum_{j=1}^{\infty}
  \lambda_j^{-1} \int_0^T |\beta_j(t)|^2 \,dt.
\end{align}
Therefore $f$ is in $L_2(0,T;V^*)$ whenever $\int_0^T|\beta_j|^2\,dt\leq C\lambda_j |\alpha_j|^2$ holds
eventually for some $C>0$, and here a direct calculation gives
\begin{equation}
  \int_0^T\frac{|\beta_j|^2}{|k_j|^{2}}\,dt=
  \frac{e^{2T(\sqrt{\lambda_j}-\lambda_j)}-e^{2\theta_j
      T(\sqrt{\lambda_j}-\lambda_j)}}{2\sqrt{\lambda_j}-2\lambda_j}
  =\frac{e^{2T\sqrt{\lambda_j}}
    (e^{2T(1-\theta_j)(\lambda_j-\sqrt{\lambda_j})}-1)}{2e^{2T\lambda_j}(\lambda_j-\sqrt{\lambda_j})}.
\end{equation}

So in view of the expression for $k_j$, the quadratic integrability of $f$ follows if the
$\theta_j$ can be chosen so that the above numerator is estimated by
$C(\lambda_j-\sqrt{\lambda_j})(e^{T\sqrt{\lambda_j}}-e^{\theta_j T\sqrt{\lambda_j}})^2$
with $C$ independent of $j\ge J$ for a suitable $J$, or more simply if
\begin{equation}
  e^{2T(1-\theta_j)(\lambda_j-\sqrt{\lambda_j})}-1
\leq C(\lambda_j-\sqrt{\lambda_j})(1-e^{-(1-\theta_j) T\sqrt{\lambda_j}})^2.
\end{equation}
We may take $J$ so that $\lambda_j >3$ for all $j\ge J$, since at most finitely
many eigenvalues do not fulfill this.
Then $\theta_j:=1-(\lambda_j-\sqrt{\lambda_j})^{-1}$ belongs to $\,]0,1[\,$, and the above
is reduced to
\begin{equation}
  \label{eq:asymp-condition}
  \exp({2T})-1 \leq C(\lambda_j-\sqrt{\lambda_j})(1-\exp(-\frac{T}{\sqrt{\lambda_j}-1}))^2.
\end{equation}
Applying the Mean Value Theorem to $\exp$ on $[-\frac{T}{\sqrt{\lambda_j}-1},0]$, we
obtain the inequality 
\begin{equation}
  (\lambda_j-\sqrt{\lambda_j})(1-\exp(\frac{-T}{\sqrt{\lambda_j}-1}))^2
  \geq 
  \exp(-\frac{2T}{\sqrt{3}-1})\frac{T^2\lambda_j}{\lambda_j-\sqrt{\lambda_j}}
  >\exp(-{4T})T^2>0.
\end{equation}
Hence  \eqref{eq:asymp-condition} is fulfilled for $C=\exp({6T})/T^2$.
\end{proof}

\begin{Remark} \label{yf-rem}
  In the above proof $\supp\beta_j\subset[\theta_j T,T]$, so the given $v$ can
  be attained by $y_f$ by arranging the coefficients $\beta_j$ in each dimension 
  \emph{successively} as time approaches $T$,
  as $\theta_j\nearrow1$ follows in \eqref{betaj-id} by counting the eigenvalues so that $\lambda_j\nearrow\infty$. 
  This can even be postponed to any given $T_0<T$, for $\supp \beta_j\subset [T_0,T]$ holds whenever
  $\theta_j T\ge T_0$, and we may reset to $\theta_j=T_0/T$ and adjust the $k_j$ accordingly, 
  for the finitely many remaining $j$.
  Both themes may be of interest in infinite dimensional control theory.
\end{Remark}

In order to isolate $u(0)$  in \eqref{eq:u_T}, it will of course be decisive that the
operator $e^{-TA}$ has an inverse, as was shown for general analytic semigroups in
Proposition~\ref{inj-prop}. 

For our Lax--Milgram operator $A$ with analytic semigroup $e^{-tA}$ generated by $\Gen=-A$, 
it is the symbol $e^{tA}$ that denotes the inverse, consistent with the sign convention in \eqref{eq:inverse}. 
Hence the properties of $e^{tA}$ can be read off from Proposition~\ref{inverse-prop}, where
\eqref{comm-eq} gives
\begin{equation}
  \label{eq:LMinclusion}
 e^{-tA}e^{TA}\subset e^{(T-t)A} \qquad\text{ for $0\le t\le T$}.
\end{equation}

Moreover, it is decisive for the interpretation of the compatibility conditions in
Section~\ref{wp-ssect} below to know that the domain inclusions in Proposition~\ref{inverse-prop} are
strict. We include a mild sufficient condition along with a characterisation of the domain $D(e^{tA})$.
 
\begin{Proposition} \label{domain-prop}
If $H$ has an orthonormal basis of eigenvectors $(e_j)_{j\in\N}$ of $A$ so that the
corresponding eigenvalues fulfil $\Re\lambda_j \rightarrow \infty$ for $j\to \infty$, then the 
inclusions in \eqref{eq:domain_inclusion} are both \emph{strict}, and $D(e^{tA})$ is the
completion of $\Span(e_j)_{j\in\N}$ with respect to the graph norm, 
\begin{align}\label{eq:graph_norm_onb}
  \|x \|^2_{D(e^{tA})} = \sum_{j=1}^{\infty} (1+ e^{2\Re\lambda_j t})|\ip{x}{e_j}|^2.
\end{align}
The domain $D(e^{tA})$ equals the subspace $S\subset H$ in which the right-hand side
is finite.
\end{Proposition}

\begin{proof}
If $x \in S$ the vector 
$v =  \sum_{j=1}^{\infty} e^{\lambda_j t} \ip{x}{e_j} e_j$  is well defined in $H$, and 
with methods from Remark~\ref{rem:bijektiv_onb} it follows that $e^{-tA}v=x$; i.e.\ 
$x\in D(e^{tA})$.

Conversely, for $x \in D(e^{tA})$ there is a vector $y \in H$ such that 
$x = e^{-tA}y = \sum_{j=1}^{\infty} \ip{y}{e_j} e^{-t\lambda_j} e_j$.
That is, $e^{\lambda_j t} \ip{x}{e_j} = \ip{y}{e_j} \in \ell_2  $, so $x \in S$.
Then $|e^{tA}x|^2=\sum e^{2\Re\lambda_j t}|\ip{x}{e_j}|^2$ yields
\eqref{eq:graph_norm_onb}.

Now any $x \in D(e^{t'A})$ is also in $D(e^{tA})$ for $t<t'$, 
since $\Re\lambda_j > 0$ holds in \eqref{eq:graph_norm_onb} for all $j$ by $V$-ellipticity.
As $\Re\lambda_j \rightarrow \infty$,  we may choose a subsequence so that 
$\Re\lambda_{j_n}> n$ and set
\begin{align}
  x = \sum_{n=1}^{\infty} \frac1{n} e^{-\lambda_{j_n}t} e_{j_n}.
\end{align}
Here $x\in D(e^{t A})$ as it is in $S$ by construction for $t\ge0$;
but not in $D(e^{t' A})$ for $t'>t$ as
\begin{align}
  \sum_{j=1}^{\infty} e^{2\Re \lambda_j t' } |\ip{x}{e_j}|^2 =  
  \sum_{n=1}^{\infty} e^{2\Re \lambda_{j_n} (t'- t) }\frac1{n^2} > 
  \sum_{n=1}^{\infty} \frac{e^{2n (t'- t) }}{n^2} = \infty .
\end{align}

Furthermore, using orthogonality, it follows
for any $x \in D(e^{tA})$ that, for $N \rightarrow \infty$,
\begin{align}
 \big\|x - \sum_{j\leq N} \ip{x}{e_j} e_j\big\|_{D(e^{tA})}^2 
  =\sum_{J>N} (1+ e^{2\Re\lambda_j t})|\ip{x}{e_j}|^2
  \rightarrow 0 .
\end{align}
Hence the space $D(e^{tA})$ has $ \Span(e_j)_{j\in\N}$ as a dense subspace.
That is, the completion of the latter with respect to the graph norm identifies with the former.
\end{proof}

After this study of the map $y_f$, the injectivity of the operator $e^{-tA}$ and the domain
$D(e^{tA})$, cf.\ Propositions~\ref{inj-prop}, \ref{inverse-prop}, \ref{yf-prop} and \ref{domain-prop}, 
we address the final value problem
\eqref{eq:fvp} by solving \eqref{eq:u_T} for the vector $u(0)$. This is done by considering the map
\begin{align} 
  u(0) \mapsto  e^{-TA} u(0) +y_f.
\end{align}
This is composed of the bijection $e^{-TA}$ and a translation by the vector $y_f$, 
hence is bijective from $H$ to the affine space $R(e^{-TA}) + y_f$. 
In fact, using \eqref{eq:inverse}, inversion gives
\begin{align}   \label{eq:u_0_from_u_T}
   u(0) = e^{TA}\bigl(u(T) - \int_0^T e^{-(T-s)\A}f(s) \,ds\bigr) = e^{TA}(u(T) - y_f).
\end{align}

This may be summed up thus:

\begin{Theorem}\label{thm:bijection}
For the set of solutions $u$ in $X$ of the differential equation $(\partial_t +\A)u=f$ with 
fixed data $f \in L_2(0,T; V^*)$, the formulas \eqref{eq:u_T} {and} \eqref{eq:u_0_from_u_T} 
give a bijective correspondence between the initial states $u(0)$ in $H$ 
and the terminal states $u(T)$ in $y_f +D(e^{TA})$. 
\end{Theorem}

In view of the linearity, the affine space $y_f+D(e^{TA})$ might seem surprising. However, a suitable
reinterpretation gives the compatibility condition introduced in the next section.

\subsection{Well-Posedness of the Final Value Problem}
\label{wp-ssect}

Since $R(e^{TA})\subset H$, the initial state in \eqref{eq:u_0_from_u_T} can be inserted into formula
\eqref{u-id}, so any solution $u$ of \eqref{eq:fvp} must satisfy
\begin{align}  \label{eq:fvp_solution} 
   u(t) = e^{-tA} e^{TA}(u_T - y_f) + \int_0^t e^{-(t-s)\A}f(s) \,ds.
\end{align}
Here one could contract the first term a bit, as $e^{-tA} e^{TA} \subset e^{(T-t)A}$ by \eqref{eq:LMinclusion}.
But we refrain from this because $e^{-tA} e^{TA}$ rather obviously applies to $u_T-y_f$ if and only if
this vector belongs to $D(e^{TA})$\,---\,and the following theorem corroborates that this is
equivalent to the unique solvability in $X$ of the final value problem \eqref{eq:fvp}:

\begin{Theorem}\label{fvp-thm}
Let $V$ be a separable Hilbert space contained algebraically, topologically and densely in $H$, and
let $A$ be the Lax--Milgram operator defined in $H$ from a bounded $V$-elliptic sesquilinear form $a$, and having 
bounded extension $\A\colon V \rightarrow V^*$. For given $f\in L_2(0,T;V^*)$ and $u_T \in H$, the condition 
\begin{align}  \label{eq:comp_cond}
  u_T - y_f \in D(e^{TA})
\end{align}
is necessary and sufficient for the existence of some  $u\in X$, cf.~\eqref{eq:X},
that solves the final value problem \eqref{eq:fvp}.
Such a function $u$ is uniquely determined and given by \eqref{eq:fvp_solution}, where all terms
belong to $X$ as functions of $t$.
\end{Theorem}

\begin{proof}
When \eqref{eq:fvp} has a solution $u \in X$, then $u_T$ is reachable from
the initial state $u(0)$ determined from the bijection in Theorem~\ref{thm:bijection}, which 
gives that $u_T-y_f = e^{-TA} u(0)\in D(e^{TA})$. Hence \eqref{eq:comp_cond}
is necessary and \eqref{eq:fvp_solution} follows by insertion, as explained prior to
\eqref{eq:fvp_solution}. Uniqueness is obvious from the right-hand side of \eqref{eq:fvp_solution}. 

When $u_T$, $f$ fulfill \eqref{eq:comp_cond},
then $u_0 = e^{TA}(u_T - y_f)$ defines a vector in $H$, so 
Theorem~\ref{thm:Temam} yields a
function $u\in X$ solving $(\partial_t +\A)u = f$ and $u(0)=u_0$. 
According to Theorem~\ref{thm:bijection} 
this $u$ has final state $u(T)=e^{-TA}e^{TA}(u_T-y_f)+y_f=u_T$, hence solves \eqref{eq:fvp}.

Finally, the fact that the integral in \eqref{eq:fvp_solution} defines a function in $X$
follows at once from Theorem~\ref{thm:repr_for_u}, for it states that it equals the solution 
in $X$ of $\tilde u'+\A\tilde u=f$, $\tilde u(0)=0$.
Since $u\in X$ in \eqref{eq:fvp_solution}, also $e^{-tA} e^{TA}(u_T - y_f)$ is a function in $X$.
\end{proof}

\begin{Remark}\label{u(0)-rem}
When $(f,u_T)$ fulfils \eqref{eq:comp_cond}, then \eqref{eq:u_0_from_u_T}
yields that $u_T -y_f = e^{-TA} u(0)$. 
\end{Remark}

\begin{Remark}
Writing condition \eqref{eq:comp_cond} as $u_T = e^{-TA} u(0) + y_f$,
cf.~Remark~\ref{u(0)-rem}, 
this part of Theorem~\ref{fvp-thm} is natural inasmuch as each set of admissible terminal data $u_T$ are in
effect a sum of the terminal state, $e^{-TA}u(0)$, of the semi-homogeneous initial value
problem \eqref{ivp-id} with $f=0$ and of the terminal state $y_f$ of the
semi-homogeneous problem \eqref{ivp-id} with $u(0)=0$.  Moreover, the $u_T$ fill at least a dense set in $H$, as
for fixed $u(0)$ this follows from Proposition~\ref{yf-prop}; for fixed $f$ from the density of
$D(e^{TA})$ seen prior to Proposition~\ref{inverse-prop}.
\end{Remark}

\begin{Remark}   \label{PA-rem}
To elucidate the criterion $u_T- y_f \in D(e^{TA})$ in formula \eqref{eq:comp_cond} of
Theorem~\ref{fvp-thm}, we consider the matrix operator 
$P_{\A}=\left(\begin{smallmatrix}
  \partial_t+\A\\ r_T \end{smallmatrix}\right)$,
with $r_T$ denoting restriction at $t=T$, and the ``forward'' map $\Phi(f,u_T)=u_T-y_f$, which by \eqref{eq:Xnorm} and 
Proposition~\ref{yf-prop} give bounded operators 
\begin{equation}
 X \xrightarrow[~]{\quad P_{\A}\quad}
\begin{matrix}
  L_2(0,T;V^*)\\
  \oplus \\
  H
\end{matrix} 
\xrightarrow[~]{\quad \Phi \quad}  H.
\end{equation}
Then, in terms of the range $R(P_{\A})$, clearly \eqref{eq:fvp} has a solution if and only if
$
\left(\begin{smallmatrix}
f \\
u_T
\end{smallmatrix}\right)
\in R (P_{\A})
$,
so the compatibility condition \eqref{eq:comp_cond} means that 
$R(P_{\A}) = \Phi^{-1}(D(e^{TA}))=D(e^{TA} \Phi)$.
\end{Remark}

The paraphrase at the end of Remark~\ref{PA-rem} is convenient for the choice of a useful norm on the data.
Indeed, we now introduce the space of admissible data $Y = D(e^{TA} \Phi)$, i.e.
\begin{align}  \label{Y-id}
  Y = \Big\{ (f,u_T) \in L_2(0,T;V^*) \oplus H  \Bigm|  u_T - y_f \in D(e^{TA}) \Big\},
\end{align}
endowed with the graph norm on $D(e^{TA} \Phi)$ given by
\begin{align} \label{eq:normY}
  \| (f, u_T) \|_Y^2= |u_T|^2 + \|f\|^2_{L_2(0,T;V^*)} + |e^{TA}(u_T - y_f )|^2.
\end{align}
Using the equivalent norm $\vvvert\cdot\vvvert_*$ from \eqref{ipV*-id} for $V^*$, 
the above is induced by the inner product
\begin{equation}  \label{eq:ipY}
  \ip{u_T}{v_T} + \int_0^T \ip{f(s)}{g(s)}_{V^*}\,ds 
  + \ip{e^{TA}(u_T-y_f)}{e^{TA}(v_T-y_g)}.
\end{equation}
This space $Y$ is \emph{complete}: 
as $\Phi$ in Remark~\ref{PA-rem} is bounded, the composite map $e^{TA}\Phi$ is a
closed operator from $L_2(0,T;V^*)\oplus H$ to $H$, so its domain $D(e^{TA}\Phi)=Y$ is complete
with respect to the graph norm given in \eqref{eq:normY}. Hence $Y$ is a Hilbert(-able) space---but we
shall often just work with the equivalent norm on the Banach space $Y$ obtained by using simply $\|\cdot\|_*$ on $V^*$.

Moreover, the norm in \eqref{eq:normY} 
also leads to continuity of the solution operator for \eqref{eq:fvp}:

\begin{Theorem}\label{wp-thm}
The solution $u\in X$ in Theorem~\ref{fvp-thm} depends continuously on the data $(f,u_T)$ in the
Hilbert space $Y$ in \eqref{Y-id}, or  equivalently, for some constant $c$ we have 
\begin{multline} \label{wp-inequality}
  \int_0^T \|u(t)\|^2 \,dt + \sup_{t \in [0,T]} |u(t)|^2 + \int_0^T \|\partial_t u(t)\|^2_{*} \,dt
\\
  \le |u_T|^2 + c\bigg(\int_0^T\|f(t)\|^2_{*}\,dt + \big|e^{TA}(u_T - \int_0^T e^{-(T-t)\A}f(t)\,dt) \big|^2\bigg).
\end{multline}
Another equivalent norm on the Hilbert space $Y$ is obtained by omitting the term $|u_T|^2$.
\end{Theorem}

\begin{proof}
This follows from Corollary~\ref{cor:Temam} by inserting $u_0 = e^{TA}(u_T -y_f)$ from \eqref{eq:u_0_from_u_T} into
\eqref{ivp-est}, for this gives
$\|u\|^2_X  \leq c |e^{T A}(u_T -y_f)|^2 + c\|f\|^2_{L_2(0,T;V^*)}$, where one can 
add $|u_T|^2$.
Conversely the boundedness of $y_f$ and $e^{-TA}$ yield that $|u_T|^2\le c\|f\|^2+c| e^{TA}(u_T-y_f)|^2$.
\end{proof}

Of course, Theorem~\ref{fvp-thm} and 
Theorem~\ref{wp-thm} together 
mean that the final value problem in \eqref{eq:fvp} is well posed in the spaces $X$ and $Y$.

\section{The Heat Equation With Final Data}\label{heat-sect}
To apply the theory in Section~\ref{fvp-sect}, we treat the heat equation and its final value problem.  
In the sequel $\Omega$ stands for a smooth, open bounded set in $\Rn$,
$n\ge2$ as described in \cite[App.~C]{G09}. In particular $\Omega$  is
locally on one side of its boundary $\Gamma:=\partial \Omega$.

For such sets  we consider the problem of finding the $u$ satisfying
\begin{equation}  \label{eq:heat_fvp}
\left\{
\begin{aligned}
  \partial_tu(t,x) -\Delta u(t,x) &= f(t,x) &&\text{ in } Q:= ]0,T[ \times \Omega ,
\\
  \gamma_0 u(t,x) &= g(t,x) && \text{ on } S:= ]0,T[ \times \partial \Omega,
\\
   r_T u(x) &= u_T(x) && \text{ at } \left\{ T \right\} \times \Omega.
\end{aligned}
\right .
\end{equation}
Hereby the trace of functions on $\Gamma$ is written in the operator notation $\gamma_0u=
u|_{\Gamma}$; similarly we also use $\gamma_0$ for traces on $S$. $r_T$ denotes the trace 
operator at $t=T$.

We shall also use $H^1_0(\Omega)$, which is the subspace obtained by closing 
$C_0^\infty(\Omega)$ in the Sobolev space $H^1(\Omega)$. Dual to this one has $H^{-1}(\Omega)$,
which identifies with the set of restrictions to $\Omega$ from $H^{-1}(\Rn)$, endowed with the
infimum norm. The reader is referred to Chapter 6 and Remark 9.4 in \cite{G09} for the spaces
$H^{s}(\Rn)$ and the infimum norm.

\subsection{The Boundary Homogeneous Case} \label{homoheat-ssect}
In case  $g \equiv 0$ in \eqref{eq:heat_fvp}, the consequences of the abstract results in
Section~\ref{wp-ssect} are straightforward to account for. Indeed, with 
\begin{align}
  V = H_0^1(\Omega),  \quad H = L_2(\Omega), \quad  V^* = H^{-1}(\Omega),
\end{align}
the boundary condition $\gamma_0u=0$ is imposed via the condition that $u(t)\in V$ for all $t$,
or rather through use of the Dirichlet realization of the Laplacian$-\Delta_{\gamma_0}$ 
(denoted by $\mlap_D$ in the introduction), which is
the Lax--Milgram operator $A$ induced by the triple
$(L_2(\Omega),H_0^1(\Omega),s)$ for 
\begin{align}
  s(u,v) = \sum_{j=1}^n \ip{\partial_j u}{\partial_j v}_{L_2(\Omega)}. 
\end{align}
In fact, the Poincar\' e inequality yields that $s(u,v)$ is
$H_0^1(\Omega)$-elliptic, and as it is symmetric too, $A = -\Delta_{\gamma_0}$ is a selfadjoint
unbounded operator in $L_2(\Omega)$, with $D(-\Delta_{\gamma_0})\subset H^1_0(\Omega)$.  

Hence the operator $-A = \lap_{\gamma_0}$ generates an analytic semigroup
$e^{t\lap_{\gamma_0}}$ in $\B(L_2(\Omega))$; the bounded extension 
$-\A=\lap\colon H^{1}_0(\Omega) \rightarrow H^{-1}(\Omega)$ 
induces the analytic semigroup $e^{-t\A}=e^{t\Delta}$ on
$H^{-1}(\Omega)$; cf.~Lemma~\ref{Agen-lem}.
Consistently with Section~\ref{bijection-ssect} we also set $(e^{t\Delta_{\gamma_0}})^{-1} = e^{-t\Delta_{\gamma_0}}$.

For the homogeneous problem with $g=0$ in \eqref{eq:heat_fvp} we have the solution and data spaces
\begin{align}
  X_0 &= L_2(0,T;H^1_0(\Omega)) \bigcap C([0,T]; L_2(\Omega)) \bigcap H^1(0,T; H^{-1}(\Omega)),
\label{X0-id}
\\  
  Y_0&= \left\{ (f,u_T) \in L_2(0,T;H^{-1}(\Omega)) \oplus L_2(\Omega) \Bigm|  
                  u_T - y_f \in D(e^{-T\Delta_{\gamma_0}}) \right\}.
\label{Y0-id}
\end{align}
Here,
with $y_f$ as the usual integral (cf.\ \eqref{heat-ccc} below), the data norm in \eqref{eq:normY} amounts to
\begin{align}
 \| (f,u_T) \|_{Y_0}^2  
  = \int_0^T\|f(t)\|^2_{H^{-1}(\Omega)}\,dt 
  + \int_\Omega(|u_T|^2+|e^{-T\Delta_{\gamma_0}}(u_T - y_f )|^2)\,dx.
\end{align} 

From Theorems~\ref{fvp-thm} and \ref{wp-thm} we may now read off the following result, which
is a novelty even though the problem is classical:

\begin{Theorem}  \label{heat0-thm}
Let $A=-\Delta_{\gamma_0}$ be the Dirichlet realization of the Laplacian in $\Omega$ and 
$\A=-\Delta$ its extension, as introduced above.
When $g=0$ in the final value problem \eqref{eq:heat_fvp} and 
$f \in L_2(0,T;H^{-1}(\Omega))$, $u_T \in L_2(\Omega)$, 
then there exists a solution $u$ in $X_0$ of \eqref{eq:heat_fvp} if and only if the data $(f,u_T)$
are given in $Y_0$, i.e.\ if and only if
\begin{equation}  \label{heat-ccc}
  u_T - \int_0^T e^{-(T-s)\A}f(s) \,ds\quad \text{ belongs to }\quad D(e^{-T \Delta_{\gamma_0}}). 
\end{equation}
In the affirmative case, such $u$ are uniquely determined in $X_0$ and  fulfil the estimate
$\|u\|_{X_0} \leq c \| (f,u_T) \|_{Y_0}$.
Furthermore the difference  in \eqref{heat-ccc} equals 
$e^{T\Delta_{\gamma_0}}u(0)$ in $L_2(\Omega)$. 
\end{Theorem}

\begin{Remark} \label{BE-rem}
For $A=-\Delta_{\gamma_0}$ one has the equivalent norms in Facts~1, 2 and the characterisation of 
$D(e^{-T\Delta_{\gamma_0}})$ in Proposition~\ref{domain-prop}.
This is a classical consequence of the compact embedding of
$H^1_0(\Omega)$ into $L_2(\Omega)$ for bounded sets $\Omega$ (e.g.\ \cite[Thm.~8.2]{G09}). 
Thus one obtains for $f=0$, $g=0$ the situation described in the introduction, 
where the space of final data, normed by $\vvvert u_T \vvvert$,
via Proposition~\ref{domain-prop} is seen to be $D(e^{-T\Delta_{\gamma_0}})$ with
equivalent norms.
As the completed solution space $\overline{\cal E}$ in the introduction one may take the Banach
space $\overline{\cal E}=X_0$, cf.\ Theorem~\ref{heat0-thm}.
\end{Remark}

\subsection{The Inhomogeneous Case}
For non-zero data, i.e., when $g\ne0$ on $S$, cf.\ \eqref{eq:heat_fvp}, one may of course try to
reduce to an equivalent homogeneous problem by choosing a function $w$ so that
$\gamma_0 w = g$ on the surface $S$. Here we recall \mbox{the classical}

\begin{Lemma}  \label{wKg-lem}
$\gamma_0\colon H^1(Q)\to H^{1/2}(S)$ is a continuous surjection having a
bounded right inverse $\tilde{K}_0$, so $w = \tilde{K}_0 g$ maps every $g\in H^{1/2}(S)$ 
to $w\in H^1(Q)$ fulfilling $\gamma_ 0w=g$ and
\begin{align}  \label{eq:wKg}
  \|w\|_{H^1(Q)} \leq c \|g\|_{H^{1/2}(S)}.
\end{align}
\end{Lemma}
Lacking a reference with details, we note that the lemma is well known for sets like $\Omega$, hence
for smooth open bounded sets $\Omega_1\subset\R^{n+1}$ with operators $\gamma_{0,\Omega_1}$ and
$\tilde K_{0,\Omega_1}$; cf.\ Theorem~B.1.9 in \cite{H} or Theorem~9.5 in \cite{G09} for the flat case. In particular, one can 
stretch $Q$ to $\,]-2T,2T[\,\times\Omega\,$ and attach rounded ends in a smooth way
 to obtain a set $\Omega_1\subset \,]-3T,3T[\,\times\Omega$ equal to
$Q$ for $0<t<T$. Here $H^1(Q)=r_QH^1(\Omega_1)$ is a classical result, when the latter space  of
restrictions to $Q$ has the infimum norm.
While $H^s(\partial\Omega_1)$ is defined using local coordinates in a
standard way, cf.\ formula {(8.10)} in \cite{G09}, the Sobolev space $H^s(S)$ on the surface $S$ can be defined
as the set of restrictions 
$r_SH^s(\partial\Omega_1)$. When $r_S\tilde g=g$, then
$\tilde K_0g=r_Q\tilde K_{0,\Omega_1}\tilde g$ defines the desired operator $\tilde K_0$, as
$\gamma_{0,\Omega_1}$ acts as $\gamma_0$ in $Q$.  

\begin{Remark} 
  \label{HsS-rem}
The norm in $H^s(S)$ can be chosen so that this is a Hilbert space; cf.\ formula {(8.10)} in \cite{G09}. 
However, Sobolev spaces on smooth surfaces is a vast subject, requiring so-called distribution
densities as explained in \cite[Sect.\ 6.3]{H}. We refer the reader to \cite[Sect.\ 8.2]{G09} for a
short introduction to this subject; as there, we prefer a more intuitive approach (exploiting the surface
measure on $\Omega_1$) but skip details. A systematic exposition of this framework can be found
in \cite[Sect.\ 4]{JoMHSi3}, albeit in a general $L_p$-setting with mixed-norms leading to
anisotropic Triebel--Lizorkin spaces $F^{s,\vec a}_{\vec p,q}(S)$ on the curved boundary, which in
general are the correct boundary data spaces for parabolic problems with different integrability
properties in space and time, as noted in \cite{JoSi08}; cf.\ the discussion of the heat equation  in \cite[Sect.\
6.5]{JoMHSi3} and the more detailed account in \cite[Ch.\ 7]{SMHphd}.
\end{Remark}

However, when splitting the solution of \eqref{eq:heat_fvp} as $u=v+w$ for $w$ as in Lemma~\ref{wKg-lem},
then $v$ should satisfy \eqref{eq:heat_fvp} with data $(\tilde f,0,\tilde u_T)$,
\begin{equation} \label{tilde-data}
  \tilde f=f-(\partial_t w-\Delta w),\qquad \tilde u_T=u_T-r_Tw.
\end{equation}
At first glance one might therefore think that $w$
is inconsequential for the compatibility condition \eqref{heat-ccc}, for 
$\tilde u_T-y_{\tilde f}$ there equals the usual term $u_T-y_f$ minus
$r_Tw-y_{\partial_t w-\Delta w}$, where the latter seemingly belongs to $D(e^{-T\Delta_{\gamma_0}})$ as
the pair $(\partial_t w-\Delta w, r_Tw)$ could seem to be a vector in the range of the operator $P_{-\Delta}$ in
Remark~\ref{PA-rem}.

But obviously this is not the case, because the function $w$ is outside the domain $X_0$ of
$P_{-\Delta}$. Indeed, $w\in L_2(0,T;H^1(\Omega))$ and has $\gamma_0 w=g\not\equiv 0$ in the
non-homogeneous case, whence $w\notin L_2(0,T;H^1_0(\Omega))$. 
So one might think it would be necessary to discuss homogeneous problems with larger
solution spaces $\tilde X_0$ than $X_0$.

We propose to circumvent these difficulties by applying Lemma~\ref{wKg-lem} to the corresponding linear
\emph{initial} value problem instead, since in the present spaces of low regularity there is no compatibility
condition needed for this: 
\begin{equation}  \label{eq:heat_ivp}
\left\{
\begin{aligned}
  \partial_tu -\Delta u &= f &&\quad\text{ in } Q,
\\
  \gamma_0 u &= g  &&\quad \text{ on } S,
\\
  r_0 u &= u_0 &&\quad \text{ at } \{ 0 \} \times \Omega.
\end{aligned}
\right .
\end{equation}
More precisely, we shall analogously to Section~\ref{fvp-sect} obtain a bijection $u(0)\leftrightarrow
u(T)$ between initial and final states by establishing a solution formula as in
Theorem~\ref{thm:repr_for_u}.  
(For general background material on \eqref{eq:heat_ivp} the reader could consult Section~III.6 in
\cite{ABHN11}, and for the fine theory including compatibility conditions we refer to \cite{GrSo90}.)

Analogously to Theorem~\ref{thm:Temam} and Corollary~\ref{cor:Temam}, we depart from well-posedness of \eqref{eq:heat_ivp}.
This is well known \emph{per se}, but we need to briefly review the explanation in order to
account later for the decisive existence of an improper integral showing up when $g\ne0$ in \eqref{eq:heat_fvp}.  

Since the solutions now take values in the full space $H^1(\Omega)$, we shall in this section denote the
solution space by $X_1$. It is given by
\begin{align}
 X_1 = L_2(0,T; H^1(\Omega)) \bigcap C([0,T];L_2(\Omega)) \bigcap H^1(0,T;H^{-1}(\Omega)),
\end{align}
and $X_1$ is a Banach space when normed analogously to \eqref{eq:Xnorm},
\begin{align} \label{Xheat-nrm}
  \| u \|_{X_1}= (\|u\|_{L_2(0,T; H^1(\Omega))}^2 + \sup_{0\leq t \leq T} \|u(t)\|_{L_2(\Omega)}^2 
                   + \|u\|_{H^1(0,T; H^{-1}(\Omega))}^2)^{1/2}.
\end{align}

As $H^1$, $H^{-1}$ are not dual on $\Omega$, the redundancy in Remark~\ref{redundancy-rem}
does not extend to the term $\sup_{[0,T]} \|u\|_{L_2}$ above.

\begin{Proposition}  \label{heat-prop}
The heat initial value problem \eqref{eq:heat_ivp} has a unique solution $u\in X_1$ for given data
$f\in L_2(0,T;H^{-1}(\Omega))$, $g\in H^{1/2}(S)$, $u_0\in L_2(\Omega)$,
and there is an estimate
\begin{align}  \label{heat-ineq}
  \|u\|_{X_1}^2 \leq 
  c ( \|u_0\|_{L_2(\Omega)}^2 + \|f\|_{L_2(0,T;H^{-1}(\Omega))}^2 + \|g\|_{H^{1/2}(S)}^2 ).
\end{align}
\end{Proposition}
\begin{proof}
With $w=\tilde{K}_0 g$ as in Lemma~\ref{wKg-lem}, we
write $u=v+w$ for some $v\in X_1$ solving \eqref{eq:heat_ivp} for data 
\begin{align}
  \tilde{f} = f - (\partial_t - \Delta)w,\qquad \tilde g=0,\qquad  \tilde{u}_0= u_0 - w(0).
\end{align} 
Here $w(0)$ is well defined, as $w \in H^1(Q)$ implies $w \in C([0,T] ; L_2(\Omega))$, by an
application of Lemma~\ref{lem:Temam}.
That $w$ even is in $X_1$ results from the easy estimates, where $I=\,]0,T[\,$,
\begin{align} \label{w-ineq}
  \| w'\|_{L_2(I;H^{-1})}^2 + \|\Delta w\|_{L_2(I;H^{-1})}^2 
  \leq \|w\|_{H^1(I;L_2)}^2 + c \|w\|_{L_2(I;H^1)}^2 
  \leq c \|w\|_{H^1(Q)}^2.
\end{align}
This moreover yields that $\tilde{f} \in L_2(0,T;H^{-1}(\Omega))$, and $\tilde{u}_0 \in L_2(\Omega)$,
so by Theorem~\ref{thm:Temam}, the boundary homogeneous problem 
for $v$ has a solution in $X_0$; cf.\ \eqref{X0-id}. 
Hence \eqref{eq:heat_ivp} has the solution $u=v+w$ in $X_1$;  and by linearity this is unique in view
of Theorem~\ref{thm:Temam}. 

Inspecting the above arguments, we first note that by \eqref{eq:addendumII}, 
\begin{align} \label{w-ineq'}
  \sup_{0\leq t \leq T} \|w(t)\|_{L_2(\Omega)} 
  \leq c (\|w\|_{L_2(0,T;L_2(\Omega))} + \|\partial_t w\|_{L_2(0,T;L_2(\Omega))}) \le c \|w\|_{H^1(Q)}^2,
\end{align}
so the estimate \eqref{w-ineq} can be sharpened to $\|w\|_{X_1}^2 \leq c \|w\|_{H^1(Q)}^2$.
Now Corollary~\ref{cor:Temam} gives
\begin{align}
  \|u\|_{X_1}^2 & \leq 2 ( \|v\|_{X_0}^2 + \|w\|_{X_1}^2)  
  \leq c (\|\tilde{u}_0\|_{L_2(\Omega)}^2 + \|\tilde{f}\|_{L_2(0,T;H^{-1}(\Omega))}^2 +
  \|w\|_{X_1}^2) 
\nonumber\\
  & \leq c (\|u_0\|_{L_2(\Omega)}^2 + \|f\|_{L_2(0,T;H^{-1}(\Omega))}^2 
  + \|(\partial_t - \Delta) w \|_{L_2(0,T;H^{-1})}^2 + \|w\|_{H^1(Q)}^2)
\end{align}
which via  \eqref{w-ineq} and \eqref{eq:wKg} entails the stated estimate \eqref{heat-ineq}.
\end{proof}

As a crucial addendum, we may apply Theorem~\ref{thm:repr_for_u} directly to the function $v$
constructed during the above proof and then substitute $v=u-w$ to derive that
\begin{equation} \label{eq:uw}
  u(t)=w(t)+ e^{t\Delta_{\gamma_0}}(u_0-w(0))+\int_0^t e^{-(t-s)\A}(f-(\partial_s-\Delta)w)\,ds.
\end{equation}
This formula for the $u$ solving the inhomogeneous final value problem applies especially for $t=T$, but we
shall keep $t$ in $[0,T]$ to deduce a formula for its solution.

Our strategy in the following will be to simplify the contributions from $w$, and ultimately to
reintroduce the boundary data $g$ instead of $w$.
To do so, we apply the Leibniz rule in
Proposition~\ref{prop:Leibniz_rule} to our function $w$ in $H^1(0,t;L_2(\Omega))$ and get 
\begin{equation}  \label{w-id}
   \partial_s(e^{(t-s)\Delta_{\gamma_0}}w(s))
   =e^{(t-s)\Delta_{\gamma_0}}\partial_sw(s)-\Delta_{\gamma_0}e^{(t-s)\Delta_{\gamma_0}}w(s).
\end{equation}
As the first inconvenience, $\Delta_{\gamma_0}$ does not commute with the semigroup,
since $w$ as an element of $H^1\setminus H^1_0$
belongs to neither the domain of the realization $-\Delta_{\gamma_0}$, nor to that of $\A$.

Secondly, the right-hand side is only integrable on $[0,t-\varepsilon]$ 
for $\varepsilon>0$, as the last term has a singularity at $s=t$; cf.\
Theorem~\ref{Pazy-thm}.  As a remedy, we may use the improper Bochner integral
\begin{align}  \label{dint-id}
  \dashint_0^t \Delta_{\gamma_0} e^{(t-s)\Delta_{\gamma_0}} w(s) \,ds 
 =\lim_{\varepsilon\to0}
  \int_0^{t-\varepsilon} \Delta_{\gamma_0} e^{(t-s)\Delta_{\gamma_0}} w(s) \,ds.
\end{align}

\begin{Lemma}
For every $w \in H^1(Q)$ the limit \eqref{dint-id} exists in $L_2(\Omega)$ and
\begin{equation}
  \label{eq:w}
  w(t)-e^{t\Delta_{\gamma_0}}w(0)
   =\int_0^t e^{(t-s)\Delta_{\gamma_0}}\partial_sw(s)\,ds
    -\dashint_0^t\Delta_{\gamma_0}e^{(t-s)\Delta_{\gamma_0}}w(s))\,ds.
\end{equation}
\end{Lemma}
\begin{proof}
As $e^{t \Delta_{\gamma_0}}$ is uniformly bounded according to Theorem~\ref{Pazy-thm} and $w\in
C([0,T],L_2(\Omega))$ was seen in the above proof, 
bilinearity gives that in $L_2(\Omega)$,
\begin{align} \label{wlim-eq}
  e^{(t-(t-\varepsilon)) \Delta_{\gamma_0}} w(t-\varepsilon) \rightarrow w(t)  
  \quad\text{ for $\varepsilon\to0$}. 
\end{align}
Moreover, integration of both sides in \eqref{w-id} gives, cf.\ Lemma~\ref{lem:Temam},
\begin{align}
  [e^{(t-s) \Delta_{\gamma_0}} w(s)]_{s=0}^{s=t-\varepsilon} 
  = \int_0^{t-\varepsilon} (-\Delta_{\gamma_0}) e^{(t-s)\Delta_{\gamma_0}} w(s) \,ds 
  + \int_0^{t-\varepsilon} e^{(t-s)\Delta_{\gamma_0}} \partial_s w(s) \,ds.
\end{align}
The left-hand side converges by \eqref{wlim-eq}, and by dominated convergence 
the rightmost term does so for $\varepsilon\to 0^+$ (through an arbitrary sequence),
so also
$\int_0^{t-\varepsilon} \Delta_{\gamma_0} e^{(t-s)\Delta_{\gamma_0}} w(s) \,ds$ converges in
$L_2(\Omega)$ as claimed. Then \eqref{eq:w} is the resulting identity among the limits.
\end{proof}

Identity \eqref{eq:w} from the lemma applies directly in the solution formula \eqref{eq:uw},
and because terms with $\partial_s w$ cancel, one obtains
\begin{equation}  \label{eq:D}
  u(t)= e^{t\Delta_{\gamma_0}}u_0 +\int_0^t e^{-(t-s)\A}f\,ds
  +\int_0^t e^{-(t-s)\A}\Delta w\,ds
  -\dashint_0^t\Delta_{\gamma_0}e^{(t-s)\Delta_{\gamma_0}}w\,ds. 
\end{equation}
We shall reduce the difference of the last two integrals in order to reintroduce the
boundary data $g$ instead of $w$.  

First we use that $\Delta=\A\A^{-1}\Delta$ on $H^1(\Omega)$ and write both terms as
improper integrals, 
\begin{equation}
   \label{eq:wdiff}
   -\dashint_0^t \A e^{-(t-s)\A}(I-\A^{-1}\Delta) w(s)\,ds.
\end{equation}
Here $Q=I-\A^{-1}\Delta$ is a well-known projection from the fine elliptic theory of the problem
\begin{equation}
  -\Delta u= f,\quad \gamma_0 u=g.
\end{equation}
In fact, if this is treated via the matrix operator 
$\left(\begin{smallmatrix} \mlap\\ \gamma_0\end{smallmatrix}\right)$,
which has an inverse in row form 
$\left(\begin{smallmatrix} -\A^{-1}& K_0\end{smallmatrix}\right)$ that applies to the data
$\left(\begin{smallmatrix} f\\ g\end{smallmatrix}\right)$,
the basic composites appear in the two operator identities on $H^1(\Omega)$ and
$H^{-1}(\Omega)\oplus H^{1/2}(\Gamma)$ respectively,
\begin{align}
  I &= \begin{pmatrix}  -\A^{-1}& K_0\end{pmatrix}
   \begin{pmatrix} -\Delta\\ \gamma_0\end{pmatrix}
    =\A^{-1}\Delta+K_0\gamma_0,
\label{Dir1-id}
\\
  \begin{pmatrix} I&0\\ 0& I\end{pmatrix}
  &=\begin{pmatrix} -\Delta\\ \gamma_0\end{pmatrix}\begin{pmatrix}  -\A^{-1}& K_0\end{pmatrix}
   =\begin{pmatrix}  \Delta\A^{-1}& \Delta K_0\\ -\gamma_0\A^{-1}& \gamma_0K_0\end{pmatrix}.
\label{Dir2-id}
\end{align}
Thus we get from the first formula that $Q=I-\A^{-1}\Delta=K_0\gamma_0$ on $H^1(\Omega)$. 

However, before we implement this, we emphasize that the simplicity of the formulas \eqref{Dir1-id} and~\eqref{Dir2-id}
relies on a specific choice of $K_0$ explained in the following:
 
As $\A=\Delta\big|_{H^1_0}$ holds in the distribution sense, $P:= \A^{-1} \Delta $ clearly fulfils
$P^2=P$, is bounded $H^1\to H_0^1$ and equals $I$ on $H^1_0$, so $P$ is the projection onto
$H_0^1(\Omega)$ along its null space, which evidently is the closed subspace of harmonic
$H^1$-functions, namely
\begin{equation}
  Z(-\Delta)=\{\, u\in H^1(\Omega)\mid -\Delta u=0\,\}. 
\end{equation}
Therefore $H^1$ is a direct sum,
\begin{align}  \label{HZ-eq}
  H^1(\Omega) = H_0^1(\Omega) \dotplus Z(-\Delta).
\end{align}
We also let $Q = I-P$ denote the projection on $Z(-\Delta)$ along $H_0^1(\Omega)$, as from the
context it can be distinguished from the time cylinder (also denoted by $Q$).

Since $\gamma_0\colon H^1(\Omega) \rightarrow H^{1/2}(\Gamma)$ is surjective with $H^1_0$
as the null-space, it has an inverse $K_0$ on the complement $Z(-\Delta)$,
which by the open mapping principle is bounded  
\begin{align}
  K_0 \colon H^{1/2}(\Gamma) \rightarrow Z(-\Delta).
\end{align}
Hence $K_0\colon H^{1/2}(\Gamma) \rightarrow H^1(\Omega)$ is a bounded  
right-inverse, i.e.\ $\gamma_0 K_0= I_{H^{1/2}(\Gamma)}$. The rest of
\eqref{Dir2-id} follows at once. Moreover, since $\gamma_0P=0$,
\begin{align}   \label{eq:KgQ}
  K_0 \gamma_0 = K_0 \gamma_0 (P+Q) = K_0 \gamma_0 Q = I_{Z(-\Delta)} Q = Q,
\end{align}
which by definition of $Q$ and $P$ gives \eqref{Dir1-id}. ($K_0$ is known as a Poisson operator;
these are amply discussed within the pseudo-differential boundary operator calculus in \cite{G1}.)

Using this set-up we obtain:

\begin{Proposition} \label{ug-prop}
If $u$ denotes the unique solution to the initial boundary value problem \eqref{eq:heat_ivp}
provided by Proposition~\ref{heat-prop}, then $u$ fulfils the identity 
\begin{equation} \label{eq:ug}
  u(t)=e^{t\Delta_{\gamma_0}}u_0+\int_0^t e^{-(t-s)\A}f(s)\,ds
  - \dashint_0^t \A e^{(t-s)\Delta_{\gamma_0}}K_0 g(s)\,ds,
\end{equation}
where the improper integral converges in $L_2(\Omega)$ for every $t\in [0,T]$.  
\end{Proposition}
\begin{proof}
Because of \eqref{eq:KgQ} we may write $(I-\A^{-1}\lap)w=Q w=K_0\gamma_0 w= K_0 g$ when $\gamma_0 w=g$,
and when this is applied in \eqref{eq:wdiff}, the solution formula \eqref{eq:D} simplifies to \eqref{eq:ug}.
\end{proof}

For $t=T$ the second term in \eqref{eq:ug} gives back $y_f=\int_0^T e^{-(T-s)\A}f(s)\,ds$
from Section~\ref{fvp-sect}.
However, the \emph{full} influence on $u(T)$ from the boundary data $g$ is collected in the third term as
\begin{align} \label{zg-id}
  z_g = \dashint_0^T \A e^{(T-s) \Delta_{\gamma_0}} K_0 g(s) \,ds.
\end{align}
That the map $g \mapsto z_g$ is well defined is clear by taking $t=T$ in Proposition~\ref{ug-prop};
this is a non-trivial result. 
The map is linear by the calculus of limits. 
In case $f=0$, $u_0=0$ it is seen from \eqref{eq:ug} that
$z_g = u(T)$, so obviously $\|z_g\|_{L_2(\Omega)}\leq \sup_t\|u(t)\|_{L_2(\Omega)}$, 
which in turn is estimated by $c \|g\|_{H^{1/2}(S)}$ using Proposition~\ref{heat-prop}. 
This proves

\begin{Lemma} \label{zg-lem}
  The linear operator $g\mapsto z_g$ is bounded $H^{1/2}(S)\to L_2(\Omega)$.
\end{Lemma}

Finally, from Proposition~\ref{ug-prop}, we conclude for an arbitrary solution in $X_1$ of the heat
equation $u'-\Delta u=f$ with $\gamma_0 u=g$ on $S$ that 
\begin{equation} \label{uTu0-id}
  u(T)=e^{T\Delta_{\gamma_0}}u(0)+y_f-z_g.  
\end{equation}
Therefore we also here have a bijection $u(0)\leftrightarrow u(T)$,
for the above breaks down to application of the bijection $e^{T\Delta_{\gamma_0}}$, cf.\
Proposition~\ref{inj-prop}, and a translation  {in $L_2(\Omega)$}  by the fixed vector $y_f-z_g$. 

We are now ready to obtain unique solvability of the inhomogeneous final value \mbox{problem \eqref{eq:heat_fvp}}.
Our result for this is similar to the abstract Theorem~\ref{fvp-thm} (as is its proof),
except for the important clarification that the boundary data $g$ 
\emph{do} appear in the compatibility condition, but only via the term $z_g$:

\begin{Theorem}  \label{yz-thm}
For given data $f \in L_2(0,T; H^{-1}(\Omega))$, $g \in H^{1/2}(S)$,  
$u_T \in L_2(\Omega)$ 
the final value problem \eqref{eq:heat_fvp} is solved by a function $u\in X_1$, whereby 
\begin{equation} 
   X_1= L_2(0,T; H^1(\Omega)) \bigcap C([0,T];L_2(\Omega)) \bigcap H^1(0,T; H^{-1}(\Omega)),
\end{equation}
if and only if the data in terms of \eqref{yf-id} and \eqref{zg-id} satisfy the compatibility
condition 
\begin{equation}  \label{yz-cnd}
  u_T - y_f + z_g \in D(e^{-T \Delta_{\gamma_0}}).
\end{equation}
In the affirmative case, $u$ is uniquely determined in $X_1$ and has the representation
\begin{align}\label{yz-id}
  u(t) = 
   e^{t \Delta_{\gamma_0}} e^{-T \Delta_{\gamma_0}} (u_T - y_f + z_g) 
   + \int_0^t e^{(t-s)\Delta} f(s) \,ds - \dashint_0^t \Delta e^{(t-s) \Delta_{\gamma_0}} K_0 g(s) \,ds,
\end{align}
where the three terms all belong to $X_1$ as functions of $t$.
\end{Theorem}

\begin{proof}
Given a solution $u\in X_1$, the bijective correspondence yields
$u_T=e^{T\Delta_{\gamma_0}}u(0)+y_f-z_g$, so that \eqref{yz-cnd} necessarily holds. 
Inserting its inversion
$u(0)=e^{-T\Delta_{\gamma_0}}(u_T-y_f+z_g)$ into the solution formula from
Proposition~\ref{ug-prop} yields \eqref{yz-id}; thence uniqueness of $u$.

If \eqref{yz-cnd} does hold, $u_0=e^{-T\Delta_{\gamma_0}}(u_T-y_f+z_g)$ is a vector in
$L_2(\Omega)$, so the initial value problem with data $(f,g,u_0)$ can be solved by means of
Proposition~\ref{heat-prop}. Then one obtains a function $u\in X_1$ that also solves the final value
problem \eqref{eq:heat_fvp}, since in particular $u(T)=u_T$ is satisfied, cf.\  the
{bijection}~\eqref{uTu0-id} and the definition of $u_0$. 

The final regularity statement follows from the fact that $X_1$ also is the solution space for the
initial value problem in Proposition~\ref{heat-prop}. Indeed, even the improper integral is a
solution in $X_1$ to \eqref{eq:heat_ivp} with data $(f,g,u_0)=(0,g,0)$, according to
Proposition~\ref{ug-prop}; cf.\ the proof of Lemma~\ref{zg-lem}. 
Similarly the integral containing $f$ solves an initial value problem with
data $(f,0,0)$, hence is in $X_1$. In addition, the first term \mbox{in \eqref{yz-id}} solves
\eqref{eq:heat_ivp} for data 
$(0,0,e^{-T \Delta_{\gamma_0}} (u_T - y_f + z_g))$.
\end{proof}

We let $Y_1$ stand for the set of admissible data. 
Within $L_2(0,T;H^{-1}(\Omega))\oplus H^{1/2}(\Gamma) \oplus L_2(\Omega)$ it is the subspace given, 
via the map $\Phi_1(f,g,u_T)= u_T-y_f+z_g$, as
\begin{align}
  Y_1= \Big\{(f,g,u_T) \Bigm| u_T - y_f + z_g \in D(e^{-T \Delta_{\gamma_0}}) \Big\}
     =D(e^{-T \Delta_{\gamma_0}} \Phi_1).
\end{align}

Correspondingly we endow $Y_1$ with the graph norm of {the operator} $e^{-T \Delta_{\gamma_0}} \Phi_1$, 
that is, of the composite map $(f,g,u_T)\mapsto e^{-T \Delta_{\gamma_0}}(u_T-y_f+z_g)$.
Again, $e^{-T\lap_D}\Phi_1(f,g,u_T)$ equals the initial state $u(0)$ steered by $f$, $g$ 
to the final state $u(T)=u_T$, as is evident for $t=0$ in \eqref{yz-id}.

Recalling that $\A=\mlap\colon H^{1}_0(\Omega)\to H^{-1}(\Omega)$, the above-mentioned graph norm is given by
\begin{multline}  \label{Yheat-nrm}
  \| (f,g,u_T)\|_{Y_1}^2 =  \|u_T\|^2_{L_2(\Omega)} + \|g\|^2_{H^{1/2}(Q)} +
  \|f\|^2_{L_2(0,T;H^{-1}(\Omega))}
\\
   + \int_\Omega\Big|e^{-T\Delta_{\gamma_0}}\Big(u_T - \int_0^T\!e^{-(T-s)\A}f(s)\,ds + 
   \dashint_0^{T}\! \A e^{(T-s)\Delta_{\gamma_0}} K_0 g(s) \,ds\Big)\Big|^2\,dx.
\end{multline} 
Here the last term is written with explicit integrals to emphasize the
complexity of the fully inhomogeneous boundary and final value problem \eqref{eq:heat_fvp}. 

Completeness of $Y_1$ follows from continuity of $\Phi_1$, cf.\  Lemma~\ref{zg-lem} concerning
$z_g$. Indeed, its composition to the left with the closed operator $e^{-T\Delta_{\gamma_0}}$ in $L_2(\Omega)$ (cf.\
Proposition~\ref{inverse-prop}) is also closed. Hence its domain
$D(e^{-T\Delta_{\gamma_0}} \Phi_1)=Y_1$
is complete with respect to the graph norm in \eqref{Yheat-nrm}. As this norm is
induced by an inner product when the norm of $H^{-1}(\Omega)$ is taken as $\vvvert\cdot\vvvert_*$ from
\eqref{ipV*-id}, and when $H^{1/2}(Q)$ is normed as in Remark~\ref{HsS-rem}, $Y_1$ is a Hilbert(-able) space.  

Analogously to the proof of Theorem \ref{wp-thm}, continuity of 
$(f,g,u_T)\mapsto u$ is now seen at once by inserting the
expression  $u_0 = e^{-T \Delta_{\gamma_0}} (u_T - y_f +z_g)$ from  \eqref{uTu0-id} into the
estimate in Proposition~\ref{heat-prop}. Thus we obtain:

\begin{Corollary} \label{yz-cor}
The unique solution $u$ of problem \eqref{eq:heat_fvp} lying in the Banach space $X_1$ 
depends continuously on the data  $(f,g,u_T)$ in the Hilbert space $Y_1$, when these are given 
 the norms in \eqref{Xheat-nrm} and \eqref{Yheat-nrm}, respectively.
\end{Corollary}

Taken together, Theorem~\ref{yz-thm} and Corollary~\ref{yz-cor} yield that the fully inhomogeneous
final value problem \eqref{eq:heat_fvp} for the heat equation is well posed in the spaces $X_1$ and $Y_1$.

\section{Final Remarks} \label{final-sect}\vspace{-9pt}
\subsection{Applicability}
For the special features of final value problems for Lax--Milgram operators $A$, 
it is of course decisive to have a proper subspace $D(e^{TA})\subsetneq H$,
for if $D(e^{TA})$ fills $H$ the compatibility condition \eqref{eq:comp_cond} will be
redundant---and \eqref{eq:comp_cond} moreover only becomes stronger as the terminal time $T$
increases, \mbox{if $D(e^{TA})$} decreases with larger $T$.

Within semigroup theory on a Banach space $B$, the above  means that the ranges $R(e^{t\Gen})$ should
form a \emph{strictly} descending chain of inclusions in the sense that, for $t'>t>0$,
\begin{align} \label{incl'-eq}
  R(e^{t'\Gen}) \subsetneq R(e^{t\Gen})\subsetneq B .
\end{align}

Non-strictness is here characterised by the rather special spectral
properties of $\Gen$ in (iv): 

\begin{Theorem} \label{incl-thm}
  For a $C_0$-semigroup $e^{t\Gen}$ with $\|e^{t\Gen}\|\le M e^{\omega t}$ the following are
  equivalent:
  \begin{itemize}
  \item[{\rm (i)}] $e^{t\Gen}$ is injective  and 
       $R(e^{t' \Gen})= R(e^{t\Gen})$ holds for some $t,t'$ with $t'>t\ge 0$.
  \item[{\rm (ii)}] $e^{t\Gen}$ is injective with range $R(e^{t\Gen})=B$ for every $t\ge 0$.
  \item[{\rm (iii)}] The semigroup is embedded into a $C_0$-group $G(t)$ satisfying
        $\|G(t)\|\le Me^{\omega|t|}$; 
  \item[{\rm (iv)}] The spectrum $\sigma(\Gen)$ is contained in the strip in  $\C$ where
  $-\omega\le \Re \lambda\le\omega$ and 
  \begin{equation}
    \|(\Gen-\lambda)^{-n}\|\le M(|\Re\lambda|-\omega)^{-n}
    \quad\text{ for $|\Re\lambda|>\omega$, $n\in\N$}.    
  \end{equation}
  \end{itemize}
\end{Theorem}
\begin{proof}
Given (i) for $t>0$, then $R(e^{(t+\delta)\Gen})=R(e^{t\Gen})$ holds for all
$\delta\in[0,t'-t]$ in view of the \mbox{inclusions \eqref{incl-eq}}; and to
every $x\in B$ some $y$ satisfies $e^{t\Gen}e^{\delta \Gen}y=e^{t\Gen}x$, which by injectivity gives 
$x=e^{\delta \Gen}y$, so that $e^{\delta \Gen}$ is surjective for such $\delta$. 
Hence $e^{t\Gen}=(e^{(t/N)\Gen})^N$ is a 
bijection on $B$ with bounded inverse, i.e., $0\in\rho(e^{t\Gen})$.
If (i) holds for $t=0$, clearly $0\in\rho(e^{t'\Gen})$.
In both cases (ii) holds because 
$0\in\rho(e^{s\Gen})$  must necessarily hold for $s>0$ according to
\cite[Thm.~1.6.5]{Paz83}, which also states that (iii) holds. (The proof there uses
\cite[Lem.~1.6.4]{Paz83}, that can be invoked directly from (ii) since the inverse of $e^{t\Gen}$
is bounded by the Closed Graph Theorem.)
Conversely (iii) yields $R(e^{t\Gen})=R(G(t))=B$ and injectivity for all $t\ge0$, 
so (ii) and hence (i) holds.
That (iv)$\implies$(iii) is part of the content of \cite[Thm.~1.6.3]{Paz83}, which also states that
(iii) implies  (iv) for real $\lambda$, but the full statement in (iv) is then obtained from 
\cite[Rem.~1.5.4]{Paz83}. 
\end{proof}

This result is essentially known, but nonetheless given as a theorem, as it clarifies how widely the present paper applies.
Indeed, for $V$-elliptic Lax--Milgram operators $A$, the semigroups are uniformly bounded, so $\omega=0$;  
thus the strip in (iv) is the imaginary axis $\im\R$, but this is contained in 
$\rho(A)$ by Lemma~\ref{Agen-lem}.
So except in the pathological case $\sigma(A)=\emptyset$, (iv) will always be violated, as will (i) and (ii).
However, since in (i) and (ii) the operator $e^{-tA}$ is injective by Proposition~\ref{inj-prop}, 
the strict inclusions in~\eqref{incl'-eq} hold for $\Gen=-A$. This proves:

\begin{Proposition} \label{dom-prop}
  For a $V$-elliptic Lax--Milgram operator $A$ with $\sigma(A)\ne\emptyset$ there is a strictly
  descending chain of dense domains $D(e^{tA})$ of the inverses $e^{tA}=(e^{-tA})^{-1}$, i.e.\
\begin{equation} \label{incl''-eq}
  D(e^{t' A}) \subsetneq D(e^{tA})\subsetneq H\qquad\text{for $t'>t>0$} .
\end{equation}
\end{Proposition}

Therefore, for elliptic Lax--Milgram operators $A$ with non-empty spectrum, 
the compatibility \mbox{condition \eqref{eq:comp_cond}} is without redundancy, and it gets effectively stronger on longer time intervals.
Previously, these properties were verified only in a special case in Proposition~\ref{domain-prop}.

\begin{Example}
  It is illuminating to consider the final value problem on $\Rn$, for $\alpha\in\C\setminus\R$,
  \begin{equation} \label{Herbst-fvp}
    \partial_t u-\Delta u+\alpha x_1u=f,\qquad u(T)=u_T.
  \end{equation} 
At first glance this might seem to be a minor variation on the heat problem in
Section~\ref{heat-sect}, in fact just a zero-order perturbation; and notably a change to $\Omega=\Rn$.
However, interestingly it cannot be treated within the present framework: 
in a paper fundamental to analysis of the Stark effect, 
  Herbst~\cite{Her79} proved for the operator $h(\alpha)=-\Delta+\alpha x_1I$ with
  $\Im\alpha\ne0$ that the minimal realisation
  $\bar h(\alpha)$ also is maximal in $L_2(\Rn)$ with \emph{empty} spectrum, 
  \begin{equation}
    \sigma(\bar h(\alpha))=\emptyset.  
  \end{equation}
   Moreover, the numerical range of $h(\alpha)$ itself is an open, slanted halfplane 
  \begin{equation}
    \nu(h(\alpha))=\{\,z\in\C\mid \Re z > \tfrac{\Re\alpha}{\Im\alpha}\Im z \,\}.
  \end{equation}
  
  Therefore $\bar h(\alpha)$ is not sectorial, as $\nu(\bar h(\alpha))\subset
  \overline{\nu(h(\alpha))}$ shows that \eqref{Asect-id} does not hold, so existence and uniqueness for the
  forward problem cannot be derived from Theorem~\ref{thm:Temam}. 
  The fact proved in \cite{Her79} 
  that $e^{-\im t \bar h(\alpha)/\alpha}$ is a contraction semigroup, which  for $\alpha=\im$ applies to 
  $e^{-t\bar h(\im)}$ that pertains to \eqref{Herbst-fvp}, entails via the Hille--Yosida theorem
  the estimate in Theorem~\ref{incl-thm}  \mbox{(iv)} for $-\bar h(\im)$, but only for $\Re\lambda>0$. 
  Since $\Re\lambda<0$ is not covered, it is despite the empty spectrum of $\Gen=-\bar
  h(\im)=\Delta-\im x_1$ not clear whether \eqref{incl'-eq} holds with strict inclusions.
  Thus it seems open which properties final value problem \eqref{Herbst-fvp} for the Herbst operator
  $h(\alpha)$ can be shown to have.
\end{Example}

\begin{Remark} \label{GreHel-rem}
  Recently Grebenkov, Helffer and Henry \cite{GrHelHen17} studied the complex Airy operator 
  $A=-\Delta+\im x_1$  in dimension $n=1$. They considered realizations defined on $\R_+$ by 
  Dirichlet, Neumann and Robin conditions using the Lax--Milgram lemma, so results on boundary
  homogenous final value problems for $-\frac{d^2}{dx^2}+\im x$ should be straightforward to write
  down, as in Section~\ref{homoheat-ssect}. The study was extended to dimension $n=2$, under the name of
  the Bloch--Torrey operator, by Grebenkov and Helffer in \cite{GrHe17}, where bounded and
  unbounded domains with $C^\infty$ boundary was treated; in cases with non-empty spectrum there
  should be easy consequences for the associated final value problems. 
  The realisations induced by a transmission condition at an interface, which was the main theme in
  \cite{GrHelHen17, GrHe17}, are defined from a recent extension of the Lax--Milgram lemma due to
  Almog and Helffer \cite{AlHe15}, so in this case the properties of the 
  corresponding final value problems are as yet unclear.
\end{Remark} 
\begin{Remark}
  We expect that extension of the theory to certain systems of parabolic equations with prescribed boundary
  and final value data should be possible. A useful framework for the discussion of this type of
  problems could be the pseudo-differential boundary operator calculus, with matrix-formed
  operators acting in Sobolev spaces of sections of vector bundles, as described in Section 4.1 of
  \cite{G1}.
  At least the present discussion should carry over to this kind of problems when the realisations
  called $(P+G)_T$ there are variational, i.e., when they are Lax--Milgram operators for certain
  triples $(H,V,a)$; this property is analysed in great depth in Section 1.7 of \cite{G1}, to which
  we refer the interested reader. It is conceivable that the variational property is unnecessary,
  and might be avoided using the pseudo-differential boundary operator calculus, but this
  seems to require an addition to the theory of parabolic systems covered by the calculus in the form of a
  result on backward uniqueness. 
\end{Remark}

\subsection{Notes} 
Classical considerations were collected by Liebermann~\cite{Lie05} for second order parabolic differential
operators (cf.\ also Evans~\cite{Eva10}), with references back to the fundamental $L_2$-theory
including boundary points of Ladyshenskaya, Solonnikov and Uraltseva~\cite{LaSoUr68}.
A fundamental framework of functional analysis for parabolic Cauchy problems was developed by Lions
and Magenes~\cite{LiMa72}. Later a full regularity theory in scales of anisotropic $L_2$-Sobolev
spaces was worked out for general pseudo-differential parabolic problems by Grubb and Solonnikov~\cite{GrSo90},
who obtained the necessary and sufficient compatibility conditions on the data, including
coincidence for half-integer values of the smoothness; cf.~also \cite[Thm.~4.1.2]{G1}. 
This study was carried over to the corresponding anisotropic $L_p$-Sobolev spaces by
Grubb~\cite{G4}. A further extension to different integrability properties in time and space was
taken up in a systematic study of anisotropic mixed-norm Triebel--Lizorkin spaces on a time cylinder
and its flat and curved boundaries by Munch Hansen, the second author and Sickel \cite{JoMHSi3}. 
Compatibility conditions were addressed for the heat equation in mixed-norm Triebel--Lizorkin spaces
in \cite[Sect.\ 6.5]{JoMHSi3} and \cite[Ch.\ 7]{SMHphd}. In particular the latter showed that,
except for coincidence at half integer smoothness,
the recursive formulation of the compatibility conditions in \cite{GrSo90} is equivalent to the requirement
that the data belong to the null space of a certain matrix-formed operator at the curved corner $\{0\}\times\partial\Omega$.
Recent semigroup and Laplace transformation methods were exposed in \cite{ABHN11}.
Denk and Kaip~\cite{DeKa13} treated parabolic  multi-order systems via the Newton polygon and obtained
$L_p$--$L_q$ regularity results using $\cal R$-boundedness.

To our knowledge, the literature contains no previous account for 
pairs of spaces $X$ and $Y$ in which final value problems for parabolic differential equations are well posed. 

An early contribution on final value problems for the heat equation was given in 1955
by John~\cite{John55}, who dealt with numerical aspects. 
In 1961, the idea of reducing the data space to obtain
well-posedness was adopted by Miranker \cite{Miranker61}  for 
the homogeneous heat equation on $\R$, and he showed that
in the space of $L_2$-functions having compactly supported Fourier transform
there is a bijection between the initial and terminal states.

In addition to the injectivity of analytic semigroups in Proposition~\ref{inj-prop},
it is known  that $u(0)$ is uniquely determined from $u(T)$ even for $t$-dependent sesquilinear
forms $a(t;v,w)$. This was shown by Lions and Malgrange~\cite{LiMl60} with an involved argument. 
It would take us too far to quote the large amount of work on the backward uniqueness in more
loosely connected situations, often adopting the log-convexity method 
(if $|u(t)|\le |u(T)|^{t/T} |u(0)|^{1-t/T}$ then $u(T)=0$ implies $u(t)=0$ for all $t>0$, hence
$u(0)=0$ by continuity) attributed to Krein, Agmon and Nirenberg.
Instead we refer the reader to \cite{Ku07,DiNg11,DaEr18} and the references therein. 

The method of quasi-reversibility for final value problems was introduced systematically in 1967 by Latt\`es and Lions
\cite{LattesLions67}. The idea is to perturb the equation $u'+Au=0$ by adding, e.g., $-\varepsilon^2 A^2$
to obtain a well-posed problem and to derive for its solution $u_\varepsilon$ that $u_\varepsilon(x,T)$ approaches $u_T$ for
$\varepsilon\to0$, circumventing analysis of well-posedness of the original final value
problem. They assumed $f=0$ for a $V$-elliptic self-adjoint $A$. 

Showalter \cite{Sho74} addressed questions that were partly similar to ours. 
He proposed to perturb instead by $\varepsilon A\partial_t$ under the condition that
$A$ is $m$-accretive with semiangle $\theta \le\pi/4$ on a Hilbert space for $f=0$. 
He claimed uniqueness of solutions, and existence if and only if the final data
via the Yosida approximations of $-A$ allow approximation of the initial state. 
Showalter also identified injectivity of operators in analytic semigroups as an important tool.
However, his reduction had certain shortcomings; cf.\ Remark~\ref{Showinj-rem}.
In comparison we obtain the full well-posedness for general $f\ne0$ and $V$-elliptic operators of
semiangle $\theta = \operatorname{arccot}(C_3 C_4^{-1})$ belonging to the larger interval $\,]0,\pi/2[\,$.

An extensive account of the area around 1975, and of the many previous contributions using a variety of techniques, was
provided by Payne~\cite{Pay75}. A more recent exposition can be found in Chapters~2 and 3 in Isakov's
book~\cite{Isa98}, and for methods for inverse problems in general the reader may consult
Kirsch~\cite{Kir96}.

In the closely related area of exact and null controllability of parabolic problems, 
the inequality in Corollary~\ref{cor:u-inequality} is a little weaker than the observability inequality
for the full subdomain $O=\Omega$. 
In this context the role of observability was reviewed by Fernandez-Cara and Guerrero \cite{FCGu06}, emphasising 
Carleman estimates as a powerful tool in the area. A treatise on Carleman estimates in the parabolic context
was given by Koch and Tataru \cite{KoTa09}.

\section*{Acknowledgement}
The authors thank H.~Amann for his interest and comments on the literature.
Also our thanks are due to an anonymous reviewer for indicating the concise proof of Corollary~\ref{cor:u-inequality}.

\providecommand{\bysame}{\leavevmode\hbox to3em{\hrulefill}\thinspace}
\providecommand{\MR}{\relax\ifhmode\unskip\space\fi MR }
\providecommand{\MRhref}[2]{%
  \href{http://www.ams.org/mathscinet-getitem?mr=#1}{#2}
}
\providecommand{\href}[2]{#2}


\begin{thebibliography}{ABHN11}

\bibitem[ABHN11]{ABHN11}
W.~Arendt, C.~J.~K. Batty, M.~Hieber, and F.~Neubrander, \emph{Vector-valued
  {L}aplace transforms and {C}auchy problems}, second ed., Monographs in
  Mathematics, vol.~96, Birkh\"auser/Springer Basel AG, Basel, 2011.

\bibitem[AH15]{AlHe15}
Y.~Almog and B.~Helffer, \emph{On the spectrum of non-selfadjoint
  {Schr\"odinger} operators with compact resolvent}, Comm. PDE \textbf{40}
  (2015), no.~8, 1441--1466.

\bibitem[Ama95]{Ama95}
H.~Amann, \emph{Linear and quasilinear parabolic problems. {V}ol. {I}},
  Monographs in Mathematics, vol.~89, Birkh\"auser Boston, Inc., Boston, MA,
  1995, Abstract linear theory.

\bibitem[CH53]{CuHi53}
R.~Courant and D.~Hilbert, \emph{Methods of mathematical physics. {V}ol. {I}},
  Interscience Publishers, Inc., New York, N.Y., 1953.

\bibitem[CJ18]{ChJo18}
A.-E. Christensen and J.~Johnsen, \emph{On parabolic final value problems and
  well-posedness}, {C. R. Acad. Sci. Paris, Ser. I} \textbf{356} (2018),
  301--305.

\bibitem[DE18]{DaEr18}
J.~Dard{\'e} and S.~Ervedoza, \emph{Backward uniqueness results for some
  parabolic equations in an infinite rod},
  https://hal.archives-ouvertes.fr/hal-01677033 (accessed on 29 March 2018).

\bibitem[DK13]{DeKa13}
R.~Denk and M.~Kaip, \emph{General parabolic mixed order systems in {${L_p}$}
  and applications}, Operator Theory: Advances and Applications, vol. 239,
  Birkh\"auser, 2013.

\bibitem[Eva10]{Eva10}
L.~C. Evans, \emph{Partial differential equations}, second ed., Graduate
  Studies in Mathematics, vol.~19, American Mathematical Society, Providence,
  RI, 2010.

\bibitem[{Fer}06]{FCGu06}
{Fern\'andez-Cara, E. and Guerrero, S.}, \emph{Global {C}arleman inequalities
  for parabolic systems and applications to controllability}, SIAM J. Control
  Optim. \textbf{45} (2006), no.~4, 1399--1446.

\bibitem[GH16]{GrHe17}
D.~S. Grebenkov and B.~Helffer, \emph{On the spectral properties of the
  {Bloch--Torrey} operator in two dimensions}, arXiv:1608.01925, 2016.

\bibitem[GHH17]{GrHelHen17}
D.S. Grebenkov, B.~Helffer, and R.~Henry, \emph{The complex {Airy} operator on
  the line with a semipermeable barrier}, SIAM J. Math. Anal. \textbf{49}
  (2017), no.~3, 1844--1894.

\bibitem[Gru95]{G4}
G.~Grubb, \emph{Parameter-elliptic and parabolic pseudodifferential boundary
  problems in global ${L}\sb p$ {S}obolev spaces}, Math. Z. \textbf{218}
  (1995), no.~1, 43--90.

\bibitem[Gru96]{G1}
\bysame, \emph{{Functional calculus of pseudo-differential boundary problems}},
  second ed., Pro\-gress in Mathematics, vol.~65, Birkh{\"a}user, Boston, 1996.

\bibitem[Gru09]{G09}
\bysame, \emph{{Distributions and operators}}, {Graduate Texts in Mathematics},
  vol. 252, {Springer}, New York, 2009.

\bibitem[GS90]{GrSo90}
G.~Grubb and V.~A. Solonnikov, \emph{Solution of parabolic pseudo-differential
  initial-boundary value problems}, J. Differential Equations \textbf{87}
  (1990), 256--304.

\bibitem[G{\"u}n67]{Gun67}
N.~M. G{\"u}nter, \emph{Potential theory and its applications to basic problems
  of mathematical physics}, Translated from the Russian by John R.
  Schulenberger, Frederick Ungar Publishing Co., New York, 1967.

\bibitem[HD11]{DiNg11}
Dinh~Nho Hào and Nguyen~Van Duc, \emph{Stability results for backward
  parabolic equations with time-dependent coefficients}, {Inverse Problems}
  \textbf{25} (2011), (20 pp.) doi:10.1088/0266-5611/27/2/025003.

\bibitem[Her79]{Her79}
I.~W. Herbst, \emph{Dilation analyticity in constant electric field. {I}. {T}he
  two body problem}, Comm. Math. Phys. \textbf{64} (1979), 279--298.

\bibitem[H{\"o}r85]{H}
L.~H{\"o}rmander, \emph{{The analysis of linear partial differential
  operators}}, Grundlehren der mathematischen Wissenschaften, Springer Verlag,
  Berlin, 1983, 1985.

\bibitem[Isa98]{Isa98}
V.~Isakov, \emph{Inverse problems for partial differential equations}, Applied
  Mathematical Sciences, vol. 127, Springer-Verlag, New York, 1998.

\bibitem[Jan94]{Jan94}
J.~Janas, \emph{{On unbounded hyponormal operators III}}, Studia Mathematica
  \textbf{112} (1994), 75--82.

\bibitem[JHS15]{JoMHSi3}
J.~Johnsen, S.~Munch Hansen, and W.~Sickel, \emph{{Anisotropic
  {L}izorkin--{T}riebel spaces with mixed norms---traces on smooth
  boundaries}}, Math. Nachr. \textbf{288} (2015), 1327--1359. \MR{3377120}

\bibitem[Joh55]{John55}
F.~John, \emph{Numerical solution of the equation of heat conduction for
  preceding times}, Ann. Mat. Pura Appl. (4) \textbf{40} (1955), 129--142.

\bibitem[Joh00]{johnsen:00}
J.~Johnsen, \emph{{On spectral properties of Witten-Laplacians, their range
  projections and Brascamp--Lieb's inequality}}, Integr.~Equ.~Oper.~Theory
  \textbf{36} (2000), 288--324.

\bibitem[JS08]{JoSi08}
J.~Johnsen and W.~Sickel, \emph{On the trace problem for {Lizorkin--Triebel}
  spaces with mixed norms}, Math. Nachr. \textbf{281} (2008), 1--28.

\bibitem[Kir96]{Kir96}
A.~Kirsch, \emph{An introduction to the mathematical theory of inverse
  problems}, Applied Mathematical Sciences, vol. 120, Springer-Verlag, New
  York, 1996.

\bibitem[KT09]{KoTa09}
H.~Koch and D.~Tataru, \emph{{C}arleman estimates and unique continuation for
  second order parabolic equations with nonsmooth coefficients}, Comm. Partial
  Differential Equations \textbf{34} (2009), no.~4-6, 305--366.

\bibitem[Kuk07]{Ku07}
I.~Kukavica, \emph{Log-log convexity and backward uniqueness}, Proc. Amer.
  Math. Soc. \textbf{135} (2007), no.~8, 2415--2421.

\bibitem[Lie05]{Lie05}
G.~M. Lieberman, \emph{Second order parabolic differential equations}, second
  ed., World Scientific Publishing, River Edge, NJ, 2005.

\bibitem[LL67]{LattesLions67}
R.~Latt{\`e}s and J.-L. Lions, \emph{M\'ethode de quasi-r\'eversibilit\'e et
  applications}, Travaux et Recherches Math\'ematiques, No. 15, Dunod, Paris,
  1967.

\bibitem[LM60]{LiMl60}
J.-L. Lions and B.~Malgrange, \emph{Sur l'unicit{\'e} r{\'e}trograde dans les
  probl{\`e}mes mixtes parabolic}, Math. Scand. \textbf{8} (1960), 227--286.

\bibitem[LM72]{LiMa72}
J.-L. Lions and E.~Magenes, \emph{Non-homogeneous boundary value problems and
  applications. {V}ol. {I}}, Springer-Verlag, New York-Heidelberg, 1972,
  Translated from the French by P. Kenneth, Die Grundlehren der mathematischen
  Wissenschaften, Band 181.

\bibitem[LSU68]{LaSoUr68}
O.~A. Ladyzenskaya, V.A. Solonnikov, and N.N. Ural'ceva, \emph{Linear and
  quasilinear equations of parabolic type}, Translations of mathematical
  monographs, vol.~23, Amer. Math. Soc., 1968.

\bibitem[MH13]{SMHphd}
S.~Munch~Hansen, \emph{On parabolic boundary problems treated in mixed-norm
  {Lizorkin--Triebel} spaces}, Ph.D. thesis, Aalborg University; Aalborg,
  Denmark, 2013.

\bibitem[Mir61]{Miranker61}
W.~L. Miranker, \emph{A well posed problem for the backward heat equation},
  Proc. Amer. Math. Soc. \textbf{12} (1961), 243--247.

\bibitem[NP06]{NiPe06}
C.~P. Niculescu and L.-E. Persson, \emph{Convex functions and their
  applications. a contemporary approach}, CMS Books in Mathematics/Ouvrages de
  Math\'ematiques de la SMC, vol.~23, Springer, New York, 2006.

\bibitem[Pay75]{Pay75}
L.~E. Payne, \emph{Improperly posed problems in partial differential
  equations}, Society for Industrial and Applied Mathematics, Philadelphia,
  Pa., 1975, Regional Conference Series in Applied Mathematics, No. 22.

\bibitem[Paz83]{Paz83}
A.~Pazy, \emph{Semigroups of linear operators and applications to partial
  differential equations}, Applied Mathematical Sciences, vol.~44,
  Springer-Verlag, New York, 1983.

\bibitem[RS80]{ReSi80}
M.~Reed and B.~Simon, \emph{{Methods of modern mathemtical physics. I:
  Functional analysis. Revised and enlarged edition}}, Academic Press, 1980.

\bibitem[Sch66]{Swz66}
L.~Schwartz, \emph{Th{\'e}orie des distributions}, revised and enlarged ed.,
  Hermann, Paris, 1966.

\bibitem[Sho74]{Sho74}
R.~E. Showalter, \emph{The final value problem for evolution equations}, J.
  Math. Anal. Appl. \textbf{47} (1974), 563--572.

\bibitem[Tan79]{Tan79}
H.~Tanabe, \emph{Equations of evolution}, Monographs and Studies in
  Mathematics, vol.~6, Pitman, Boston, Mass., 1979, Translated from the
  Japanese by N. Mugibayashi and H. Haneda.

\bibitem[Tem84]{Tm}
R.~Temam, \emph{{Navier--Stokes equations, theory and numerical analysis}},
  Elsevier Science Publishers B.V., Amsterdam, 1984, (Third edition).

\bibitem[Yos80]{Yos80}
K.~Yosida, \emph{Functional analysis}, sixth ed., Springer-Verlag, Berlin-New
  York, 1980.

\end{thebibliography}
\end{document}